\documentclass[11pt]{article}
\usepackage{amssymb,amsmath}
\usepackage{amsthm}
\def\R{\mathbb{R}}

\newtheorem{thm}{Theorem}[section]
\newtheorem{prop}[thm]{Proposition}
\newtheorem{rem}[thm]{Remark}

\newtheorem{lem}[thm]{Lemma}

\newcommand{\dif}{\mathrm{d}}
\newcommand{\di}{\mathrm{div}}
\newcommand{\p}{\partial}
\def\P{\mathbb{P}}

\begin{document}
\title{ Smooth solutions for motion of a rigid body
 of general form in an incompressible perfect fluid
}
\author{  Yun Wang$^{\lowercase{a}}$,   Aibin Zang$^{\lowercase{b,c}}$\thanks{Corresponding Author;}\,\,}
\date{}
\maketitle

\begin{center}
$^a$ The Department of Mathematics and Statistics, McMaster
University, \\1280 Main Street West, Hamilton, ON, L8S 4K1, Canada\\
 E-mail: yunwang@math.mcmaster.ca\\

$^b$ The School of Mathematics and Computer Science, Yichun University,\\Yichun City, Jiangxi Province, 336000, P.R. China\\

$^c$ The Institute of Mathematical Sciences, The Chinese University
of Hong Kong,\\ Shatin, NT, Hong Kong\\
 E-mail: zangab05@126.com\\

\end{center}

\begin{abstract}
In this paper, we consider the interactions between a rigid body of
general form and the incompressible perfect fluid surrounding it.
Local well-posedness in the space $C([0, T); H_s)$ is obtained for
the fluid-rigid body system.


\end{abstract}

\numberwithin{equation}{section}

\numberwithin{equation}{section}

\section{Introduction}

In this paper, we investigate the motion of a solid in an
incompressible perfect fluid. The behavior of the fluid is described
by the Euler equations. The solid we consider is a rigid body, whose movement
consists of translation and rotation.  And its motion conforms
to the Newton's law. Assume that both the fluid and the rigid body
are homogeneous. For simplicity of writing, the density of the fluid
equals to 1.
The domain occupied by the solid at the time is
$\mathcal{O}(t)$, and $\Omega(t) = \mathbb{R}^3\setminus
\overline{\mathcal{O}(t)}$ is the domain occupied by the fluid.
Suppose $\mathcal{O}(0)= \mathcal{O}$ and $\Omega(0)=\Omega$ share a
smooth boundary $\partial\mathcal{O}$(or $\partial\Omega$). The
equations modeling the dynamics of the system read(see also
\cite{RR})
\begin{alignat}{12}
 &\frac{\partial u}{\partial t}+(u\cdot\nabla)u+\nabla p
=f,
&\hspace{10pt}& \text{in}\,\ \Omega(t)\times [0,T],\label{1.1}\\
&\di~ u=0, & &\text{in}\,\ \Omega(t)\times [0,T],\label{1.2}\\
&u\cdot \vec{n}=(h'+\omega\times(x-h(t)))\cdot \vec{n}, &
&\text{on}
\,\ \partial\Omega(t)\times [0,T],&&\label{1.3}\\
&\lim_{|x|\to\infty}u(x,t)=u_{\infty},\label{1.4}\\
&m h''=\int_{\partial\Omega(t)}p\vec{n} {\rm d}\sigma+f_{rb},&
&\text{in}\ [0,T],\label{1.5}\\
&(J\omega)^\prime=\int_{\partial\Omega(t)}(x-h(t))\times p\vec{n}
{\rm d}\sigma+T_{rb}, &
&\text{in}\ [0,T],\label{1.6}\\
&u(x,0)=u_0(x) & & x\in\Omega,\label{1.7}\\
&h(0)=0\in \mathbb{R}^3,\  h'(0)=l_0\in \mathbb{R}^3,\
\omega(0)=\omega_0\in \mathbb{R}^3.\label{1.8}
\end{alignat}
In the above system, $u$ and $p$ are the velocity field and the
pressure of the fluid respectively. $f$ is the external force field
applied to the fluid. $f_{rb}$ and $T_{rb}$ denote the external
force and the external torque for the rigid body respectively. $m$ is
the mass, and $J$ is the inertia matrix moment related to the mass
center of the solid. Suppose the density of the rigid body is
$\rho$, then
$$m = \int_{\mathcal{O}(t)}\rho dx = \int_{\mathcal{O}}\rho dx, $$
and
$$[J(t)]_{kl} =\int_{\bar{\mathcal{O}}(t)}\rho
\left[|x - h(t)|^2 \delta_{kl} - (x- h(t))_{k} (x- h(t))_l
\right]dx.$$
Here $h(t)$  denotes the position of the mass center of the rigid
body and $\delta_{kl}$ is the Kronecker symbol. And denote $J(0)$ by $\bar{J}$. $\omega(t)$ is the
angular velocity of the rigid body. $\vec{n}$ is the unit outward
normal to $\p\Omega(t)$.

 Assume that the center of $\mathcal{O}$ is
the origin, i.e.,
$$\int_{\mathcal{O}} ydy = 0\in \mathbb{R}^3.$$

Let $Q(t)$ be a rotation matrix associated with the angular velocity
$\omega(t)$ of the rigid body, which is the solution of the
following initial value problem:
\begin{equation}\left\{\begin{aligned}
\frac{{\rm d}Q(t)}{dt}&= A(\omega(t))Q(t)\\
Q(0)&=Id.
\end{aligned}\right.\label{1.9}
\end{equation}
Here $$A(\omega) = \left( \begin{array}{ccc}0 & -\omega_3 &
\omega_2\\
\omega_3 & 0 & -\omega_1\\
-\omega_2 & \omega_1 & 0
\end{array}\right),$$
and $Id$ is the identity matrix. Then the domain $\mathcal{O}(t)$ is
given by
$$\mathcal {O}(t)=\{Q(t)y+h(t): y\in\mathcal{O}\}.$$
  For simplicity, we assume that $f=0,
u_{\infty}=0, f_{rb}=0$ and $T_{rb}=0.$

For the case that the fluid is viscous, there have been many results
over the last two decades. The existence of global weak solutions of
the above system was proved by \cite{DE1, DE2, GLS, HS1, HS2, Ju, MST, Se}.
If the rigid body is a disk in $\mathbb{R}^2$, T.Takahashi and
M.Tucsnak\cite{TT} showed the existence and uniqueness of global
strong solutions. Later, P.Cusmille and T.Takahashi \cite{CT}
extended the result to general rigid body case in $\mathbb{R}^2$.
They also proved the local existence and uniqueness of strong
solutions in $\mathbb{R}^3$.

For the case that the fluid is inviscid, \cite{ORT2} dealt with this problem
first. When the solid is of $C^1$ and piecewise $C^2$ boundary, and the fluid fills in $\mathbb{R}^2$, a unique
global classical solution was obtained under some assumption on the
initial vorticity in \cite{ORT2}. A global weak solution was
constructed in \cite{WX2} when the initial data belongs to
$W^{1,p}$, $p>\frac43$. Recently, C. Rosier and L. Rosier\cite{RR}
proved the local existence of $W^{s,2}$-strong solutions for $d\geq
2$, $s\geq [d/2]+2$ and the solid is a ball. The key idea is to make
use of the Kato-Lai theory, which was originated from \cite{KL}.

At the same time, the fluid-rigid body system which occupies a bounded domain was also studied. In particular,
 \cite{HMT} proved the existence and uniqueness of strong solution for the three dimensional case. The approach of \cite{HMT} follows closely the idea in \cite{BB}, which is used to
study the classical Euler equations. Their method also applies to
the case that there are several solids in the fluid.

In this paper, we plan to extend the result of \cite{RR} to a more
general setting. We will study the case that the solid is smooth and of
a general form. As shown by the
system itself, it is a free boundary problem. To deal with the free boundary problem,
usually the first step is to fix the boundary. To fulfill that, \cite{RR} made a translation of
the coordinate system. However, this special transformation can only be applied to the case that
the solid is a ball. The solid in this paper is of general form, hence a different change of coordinates
from \cite{RR} should be introduced. As said before, the motion of the solid is a rigid body movement, so a natural
idea is to use a coordinate transformation consisting of the translation and
rotation of the solid. We tried in this way. As will be shown in section 2, after the transformation, the new
equivalent system has some term difficult to control.  So we gave up the idea and then applied a new transformation which
coincides with the movement of the solid in its neighborhood and becomes identity when far away from it.
The concrete form of the transformation will be given in section 2. In fact, this kind of transformation has
been used by \cite{IW, CT, DEH}.

For the new equivalent system after the transformation,  we will use the Kato-Lai theory\cite{KL} to construct a sequence of
approximate solutions and prove the limit is the solution required.  The Kato-Lai theory is a Galerkin method in spirit. One can construct the Galerkin
approximation by himself or herself. We use the theory here directly to avoid unnecessary details.
By the way, the method we apply here can also be used to deal with the
several-solids case after some minor modification.

\section{Transformed equations and main result}

As noted above, to fix the boundary,  one method is to use the coordinates
transformation. As shown in \cite{ORT1, ORT2}, a direct way is the
following one consisting of translation and rotation,
\begin{eqnarray*}
&x = Q(t)y + h(t),\ \ \ \ \ \ &\bar{u}(y,t)=Q(t)^T u(Q(t)y+h(t),t),\\
&\bar{p}(y,t)=p(Q(t)y+h(t),t),\ \ \ \ \ \ &\bar{h}(t)=\int_0^tQ(s)^T
h'(s)ds,\\ & \bar{J}=J(0),\ \ \ \ \ \ \ \ &\bar{\omega}(y,t)=Q(t)^T
\omega,
\end{eqnarray*}
where $Q(t)$ is given in section 1 and $Q(t)^T$ is the transpose of
$Q(t)$.

After the transformation, an equivalent system is obtained as
follows:
\begin{alignat}{12}
& \frac{\partial \bar{u}}{\partial
t}+[(\bar{u}-\bar{h}'-\bar{\omega}\times
y)\cdot\nabla]\bar{u}+\bar{\omega}\times\bar{u}+\nabla \bar{p} =0,
&\hspace{10pt}& \text{in}\,\ \Omega\times [0,T],\label{2.1}\\
&{\rm div}~\bar{u}=0, & &\text{in}\,\ \Omega\times [0,T],\label{2.2}\\
& \bar{u}\cdot \vec{n}=(\bar{h}'+\bar{\omega}\times y)\cdot
\vec{n}, & &\text{on}
\,\ \partial\Omega\times [0,T],&&\label{2.3}\\
&m\bar{h}''=\int_{\partial\Omega}\bar{p}\vec{n} {\rm d}\sigma-m
\bar{\omega}(t)\times\bar{h}'(t), &
&\text{in}\ [0,T],\label{2.4}\\
&\bar{J}\bar{\omega}'=\int_{\partial\Omega}y\times \bar{p}\vec{n}
{\rm
d}\sigma+\left(\bar{J}\bar{\omega}(t)\right)\times\bar{\omega}(t),&
&\text{in}\ [0,T],\label{2.5}\\
&\bar{u}(y,0)=u_0, & & y\in\Omega,\label{2.6}\\
&\bar{h}(0)=0,\,\, \bar{h}'(0)=l_0,\,\, \bar{\omega}(0)=\omega_0.
\label{2.7}
\end{alignat}

The new problem is a fixed boundary problem now. However, there is a
term $[\bar{\omega}\times y)\cdot\nabla)]\bar{u}$, whose coefficient
become unbounded at large spatial distance. For the 2D case, the
difficulty was overcome in \cite{ORT2} by assuming that $u_0$
belongs to some weighted space. The method there depends strongly on the fact
that the vorticity satisfies a transport equation in 2D.
However, the fact does not hold any more in 3D. To avoid this difficulty, we will use another change of
variables. The new transformation coincides with $Q(t)y + h(t)$ in a
neighborhood of the solid and becomes identity when far away from
it. In fact, this transformation was initially applied by Inoue and Wakimoto \cite{IW}, later  by Dintelmann etc.\cite{DEH}  and Cusmille and Takahashi \cite{CT}

More precisely, for a pair of continuous vector-valued functions $(l(t),\omega(t))$, let
\begin{equation}\label{2.8-add}h(t) = \int_0^t l(s) ds, \ \ \ t\in [0, T], \end{equation}
and
\begin{equation}\label{2.8}
V(x,t)= l(t)+\omega(t)\times(x-h(t)),\,\,
(x,t)\in\mathbb{R}^3\times[0,T], \end{equation} which is a rigid
body movement.

Choose a smooth function $\xi:\mathbb{R}^3\to\mathbb{R}$ with
compact support such that $\xi(y)=1$ in a neighborhood of
$\bar{\mathcal{O}}$, and set
$$\psi(x,t)=\xi\left(Q(t)^T(x-h(t))\right),$$
where $Q(t)$ is given by (\ref{1.9}).

Then  introduce the functions $W$ and $\Lambda$,
\begin{equation}W(x,t)=\frac12 l(t)\times(x-h(t))+\frac{|x-h(t)|^2}{2}\omega,\label{2.9}\end{equation}
\begin{equation}\Lambda(x,t) =\psi V+\nabla\psi\times W.\label{2.10}\end{equation}
 It is easy to check that $\Lambda$ satisfies the following two lemmas(or refer to \cite{CT, DEH}).

\begin{rem}In what follows you will see the function $\Lambda$ will produce a transformation which coincides with $Q(t)y +h(t)$
in the neighborhood of the solid and becomes identity when far away from it.
The reason why we use $\Lambda$ instead of $\psi V$ is that we need the  function  to be of divergence free.  The function $W$ is introduced to eliminate
the divergence of $\psi V$.
\end{rem}

\begin{lem}
\begin{itemize}
\item[(1)] $\Lambda(x,t)=0$, if $x$ is far away from $\mathcal{O}(t)$;
\item[(2)] $\Lambda(x,t)=l(t)+\omega(t)\times(x-h(t))$ in $\mathcal{O}(t)\times[0,T]$;
\item[(3)] ${\rm div}~\Lambda=0$ in $\mathbb{R}^3\times[0,T]$;
\item[(4)] For all $t\in[0,T],\Lambda(\cdot,t)$ is a $C^{\infty}(\mathbb{R}^3,\mathbb{R}^3)$
function. Moreover, for every $s\in \mathbb{N}$, $\|\Lambda(\cdot,
t)\|_{W^{s,\infty}(\mathbb{R}^3)}\le C(s, T)(|l(t)|+|\omega(t)|)
$;
\item[(5)]For all $x\in\mathbb{R}^3$, the function $\Lambda(x,\cdot)$ is in
$C^0([0,T];\mathbb{R}^3)$,
 provided that $l,\omega\in C^0[0,T]$.
\end{itemize}
\end{lem}

Next, consider the vector field $X(y,t)$ which satisfies
\begin{equation}\left\{\begin{aligned}
\frac{\partial X(y,t)}{\partial t}&=\Lambda(X(y,t),t), \,\, \ t\in(0,T],\\
X(y,0)&=y\in\mathbb{R}^3.
\end{aligned}\right.\label{2.11}
\end{equation}

\begin{lem} For every $y\in\mathbb{R}^3$, the initial-value problem (\ref{2.11}) admits
 a unique solution $X(y,\cdot):[0,T]\to\mathbb{R}^3,$ which is a $C^1$ function on $[0,T]$. Moreover,
 the solution has the following properties:\\

(1) For all $t\in[0,T]$,the mapping $y\mapsto X(y,t)$ is a
$C^\infty$ diffeomorphism  from $\mathcal{O}$  onto
$\mathcal{O}(t)$ and from $\Omega$ onto $\Omega(t)$.\\

(2) Denote by $Y(\cdot,t)$ the inverse of $X(\cdot,t).$ Then for
every $x\in\mathbb{R}^3$ the mapping $t\mapsto Y(x,t)$ is
$C^1$-continuous and satisfies the following initial value problem,
\begin{equation}\left\{\begin{aligned}
\frac{\partial Y(x,t)}{\partial t}&=-
\left[\frac{\partial X(Y(x,t),t)}{\partial y}\right]^{-1}\Lambda(x, t), \,\, \ t\in(0,T],\\
Y(x,0)&=x\in\mathbb{R}^3.
\end{aligned}\right.\label{2.12}
\end{equation}

(3) For every $x, y\in\mathbb{R}^3$ and for every $t\in[0,T]$, the
determinants of the Jacobian matrices $J_X$ of $X(y,t)$ and $J_Y$ of
$Y(x,t)$ both equal to 1, i.e.,
\begin{equation}
\det \left(J_X(y,t)\right)= \det\left(
J_Y(x,t)\right)=1.\label{2.13}
\end{equation}
\end{lem}

\vspace{3mm}Let \begin{equation}\label{2.14}\begin{array}{ll}
 x = X(y,t),\ \ \ & v(y,t)=J_Y(X(y,t),t)u(X(y,t),t),\\
 q(y,t)=p(X(y,t),t),\ \ \ & H(t) = Q(t)^T h(t),\ \ \  L(t)=Q(t)^Tl(t),
\end{array}
\end{equation}
and $R(t)$ is the vector-valued function satisfying
\begin{equation}\label{2.14-add}
 A(R(t))= Q(t)^TA(\omega(t))Q(t).
\end{equation}

Denote
\begin{equation}\label{2.15}\left\{\begin{aligned}
&G(y, t)= \left(g^{ij}(y, t)
\right)=\left(\sum_{k=1}^{3}\frac{\partial Y_i}{\partial x_k}(X(y,
t), t)
\frac{\partial Y_j}{\partial x_k}(X(y,t), t)\right),\\
&g_{ij}=\sum_{k=1}^{3}\frac{\partial X_k}{\partial y_i}(y,t)\frac{\partial X_k}{\partial y_j}(y,t),\\
& \Gamma^k_{i,j}(y,
t)=\frac{1}{2}\sum_{l=1}^{3}g^{kl}\left\{\frac{\partial
g_{il}}{\partial y_j} +\frac{\partial g_{jl}}{\partial
y_i}-\frac{\partial g_{ij}}{\partial y_l}\right\}.
\end{aligned}\right.
\end{equation}

Now one can transform the original system (\ref{1.1})-(\ref{1.8})
into the following system, which is a fixed boundary problem, (see
\cite{CT,Ko})
\begin{alignat}{12}
 &\frac{\partial v}{\partial t}+M v+N v+G\cdot\nabla q=0,
&\hspace{10pt}& \text{in}\,\ \Omega\times [0,T],\label{2.16}\\
&{\rm div}~v=0, & &\text{in}\,\ \Omega\times [0,T],\label{2.17}\\
&v(y,t)\cdot \vec{n}=(L(t)+R(t)\times y)\cdot \vec{n}, &
&\text{on}
\,\ \partial\Omega\times [0,T],&&\label{2.18}\\
&m L'(t)=\int_{\partial\Omega}q \vec{n} {\rm d}\sigma- m
R(t)\times L(t), &
&\text{in}\ \ [0,T],\label{2.19}\\
&\bar{J}R'(t)=\int_{\partial\Omega}y\times q\vec{n} {\rm
d}\sigma+\bar{J} R(t)\times R(t), &
&\text{in}\ \ [0,T],\label{2.20}\\
&v(y,0)=u_0(y), \,\, & &y\in\Omega,\label{2.21}\\ & H(0)=0,\ \ \ \
L(0)=l_0,\ \ \ \ R(0)=\omega_0.\label{2.22}
\end{alignat}
where
\begin{equation}\left\{\begin{aligned}
&(M v)_i=\sum_{j=1}^3\frac{\partial Y_j}{\partial t}\frac{\partial v_i}{\partial y_j}+\sum_{j,k=1}^3\left\{\Gamma^i_{j,k}\frac{\partial Y_k}{\partial t}+\frac{\partial Y_i}{\partial x_k}\frac{\partial^2 X_k}{\partial t\partial y_j}\right\}v_j;\\
&(N v)_i=\sum_{j=1}^3v_j\frac{\partial v_i}{\partial y_j}+\sum_{j,k=1}^3\Gamma^i_{j,k}v_j v_k;\\
&(G\cdot\nabla q)_i=\sum_{j=1}^3g^{ij}\frac{\partial q}{\partial y_j}.
\end{aligned}\right.\label{2.23}
\end{equation}

Now the main result in this paper reads
\begin{thm}
Suppose that $s> \frac52,\,\, u_0\in H^s(\Omega)$ and $u_0\cdot
\vec{n} = (l_0+\omega_0\times y)\cdot \vec{n}$ on $\p\Omega$.
Then there exist some $T_0>0$ and a unique solution $(v,q,L,R)$ of
(\ref{2.16})-(\ref{2.22}) such that $$v\in
C([0,T_0];H^s(\Omega)),\,\, \nabla q\in C([0,T_0];
H^{s-1}(\Omega)),$$ and $$L,R\in C^1([0,T_0]; \mathbb{R}^3).$$ Such
a solution is unique up to an arbitrary function of $t$ which may be
added to $q$. Furthermore, $T_0$ does not depend on $s$.
\end{thm}

\begin{rem}
 Following similar proof in section 7 for the uniqueness of the solution, we can get the stability of the solution. The method is
standard as for the classical Euler equations, so we omit the details.
\end{rem}

\begin{rem}There are still many questions for the fluid-rigid body system. For example, whether the strong solution blows up or not? If there are many solids in the fluid, is there any collision between
these solids?  Which comes first, the collision or the blow up of $\|\nabla v\|_{L^\infty}$?
\end{rem}

For readers' convenience, we list the outline of the following sections. Section 3 will be devoted to some notations or definitions for the
function spaces and some lemmas to be used in the proof of the main result. The preliminaries contain the decomposition of the function spaces associated
with this particular problem, the Kato-Lai theory for the construction of approximate solutions, and the estimates for the coordinate transformation. In section 4, we will
give the a priori estimates for the new pressure term $\nabla q$ in (\ref{2.16}), under the assumption that $\|v\|_{H^s(\Omega)\cap \tilde{X}}$ is bounded. These
estimates will be used for the uniform estimates for the approximate solutions, which are constructed in section 5. Moreover, some uniform
estimates for these solutions are derived at the same time. The convergence of these solutions is studied in section 6. And the limit is
exactly the solution required. In the last section, section 7, we prove the uniqueness of the solution. Making use of the uniqueness result, the continuity of the
solution with respect to time $t$ in some space $H_s$ is also proved.

\section{Preliminaries}

Before stating our main result, we'd like to introduce some
notations and definitions.

Suppose $\mathcal{S}$ is a domain in $\mathbb{R}^3$.
$L^2(\mathcal{S})$ is the space of $L^2$-integrable functions with
the standard inner product $(\cdot, \cdot)_{L^2(\mathcal{S})}$. By
the way, we will not distinguish the scalar function spaces and the
corresponding vector-valued function spaces.

Suppose $s$ is a nonnegative integer, then $$H^s(\mathcal{S})= \{u\in
L^2(\mathcal{S}):\ D^\alpha u \in L^2(\mathcal{S}),\ \ \forall\
\alpha\in \mathbb{N}^3, s.t. , \ |\alpha|\leq s\},$$ with the inner product
$$(u, v)_{H^s(\mathcal{S})} = \sum_{|\alpha|\leq s}(D^\alpha u, D^\alpha v)_{L^2(\mathcal{S})}.$$

The homogeneous Sobolev space
$$ D^{1,2}(\mathcal{S}) = \{u\in L_{loc}^1(\mathcal{S}):\ \nabla u \in
L^2(\mathcal{S})\},$$ with the seminorm
$$|u|_{D^{1,2}(\mathcal{S})} = \|\nabla u\|_{L^2(\mathcal{S})}.$$
If one identifies the two functions $u_1, u_2\in D^{1,2}(\mathcal{S})$
whenever $|u_1 - u_2|_{D^{1,2}(\mathcal{S})} =0$, i.e., $u_1$ and
$u_2$ differ by a constant, the quotient space
$\dot{D}^{1,2}(\mathcal{S})$, with the norm $|\cdot|_{D^{1,2}(
\mathcal{S})}$, is derived. In the following text, without any confusion, we do
not distinguish the elements in $D^{1,2}(\mathcal{S})$ and $
\dot{D}^{1,2}(\mathcal{S})$ very strictly.

Let $B_R(y)$ denote the ball centered at $y$ and with the radius
$R$. $\Omega_R:= \Omega\cap B_R(0)$. Let $\rho =
\frac{m}{|\mathcal{O}|}$, where $|\mathcal{O}|$ stands for the
volume of $\mathcal{O}$. Hence $\rho$ is the density of the solid.
Let $\tilde{X}= L^2(\mathbb{R}^3)$ be endowed with the inner
product, $$( u, v)_{\tilde{X}} = \int_\Omega u(y)\cdot v(y)dy +\rho
\int_{\mathcal{O}} u(y)\cdot v(y)dy.$$

Define $$\tilde{X}_* =\{u\in \tilde{X}:\ {\rm div}~u=0 \ \mbox{in}\
\mathbb{R}^3, \ \exists\ l, \omega\in \mathbb{R}^3, s.t.,\ u= l +
\omega\times y\ \mbox{in}\ \mathcal{O}\},$$ which is a closed
subspace of $\tilde{X}$.

\begin{rem}\ For every $u\in \tilde{X}_*$, and
suppose that $u=l + \omega\times y$ on $\mathcal{O}$. In fact, $l$ and
$\omega$ are unique vectors satisfying the above relation, and
$$l = \frac{1}{m}\int_{\mathcal{O}} \rho u dy,\ \ \ \omega=-\bar{J}^{-1}\int_{\mathcal{O}}
(u\times y)dy. $$
 The fact has been proved, see \cite{DR} or \cite{WX3}. In what follows, we will
denote the vectors $l$, $\omega$  associated with $u\in \tilde{X}_*$ by $l_u$, $\omega_u$.
\end{rem}

\vspace{2mm} Let $H_s = \{u\in \tilde{X}:\ u|_{\Omega}\in
H^s(\Omega)\}$ be endowed with the scalar product
$$(u, v)_{H_s} = (u, v)_{H^s(\Omega)}+ \rho(u,v)_{L^2(\mathcal{O})}.$$
$V_s$ is the space of functions $v\in H_s $ such that $v|_{\Omega}$
belongs to $\mathcal{D}(A)$, where $A$ is the elliptic operator $A
f=\sum_{|\alpha|\le s}(-1)^{|\alpha|}\partial^{2\alpha}f $ with
Neumann boundary conditions, and  $\mathcal{D}(A)\subseteq
H^{2s}(\Omega)$. $V_s$ is endowed with the scalar product
$$(u,v)_{V_s}=(u,v)_{H^{2s}(\Omega)}+\rho(u,v)_{L^2(\mathcal{O})}.$$

As in \cite{RR}, we introduce a bilinear form on
$V_s\times\tilde{X}$:
$$\langle v,u\rangle=\left(\sum_{|\alpha|\le
s}(-1)^{|\alpha|}\partial^{2\alpha}v,u\right)_{L^2(\Omega)}+\rho(v,u)_{L^2(\mathcal{O})}.
$$

Since the main idea in this paper is the Kato-Lai theory, so we'd like to give a brief description of
the theory, which is cited from \cite{RR}. For more details, please refer to \cite{KL}. Let $V, H, X$ be three real separable Banach
 spaces. We say that $\{V, H, X\}$ is an admissible triplet if the following conditions hold.
\begin{itemize}
 \item $V\subset H \subset X$, and the inclusions are dense and continuous.

\item $H$ is a Hilbert space, with inner product $(\cdot, \cdot)_H$ and norm $\|\cdot \|_H
=(\cdot, \cdot)_{H}^{\frac12}.$

\item There is a continuous, nondegenerate bilinear form on $V\times X$, denoted by $\langle \cdot, \cdot\rangle$,
such that
$$\langle v, u\rangle = (v, u)_H\ \ \ \mbox{for all}\ \ v\in V, u\in H.$$
\end{itemize}

Recall that the bilinear form $\langle v, u \rangle $ is continuous and nondegenerate when
\begin{equation}
|\langle v, u\rangle |\leq C\|v\|_{V}\|u\|_{X} \ \ \ \mbox{for some constant }\ C>0;
\end{equation}
\begin{equation}
 \langle v, u\rangle =0 \ \ \mbox{for all }\ \ u\in X \ \mbox{implies}\ v=0;
\end{equation}
\begin{equation}
 \langle v, u\rangle =0\ \ \ \mbox{for all}\ \ v\in V\ \ \mbox{implies} \ u=0.
\end{equation}

A map $A:[0, T]\times H\rightarrow X$ is  said to be sequentially weakly continuous if $A(t_n, v_n)
\rightharpoonup A(t,v)$ in $X$ whenever $t_n \rightarrow t$ and $v_n \rightharpoonup  v$ in $H$. We denote
by $C_w([0, T]; H)$ the space of sequentially weakly continuous functions from $[0, T]$ to $H$, and by
$C_w^1([0, T]; X)$ the space of the functions $u\in W^{1, \infty}(0, T; X)$ such that $\frac{du}{dt}
\in C_w([0, T], X).$

We are concerned with the Cauchy problem
\begin{equation}\label{KL1}\frac{dv}{dt} + A(t,v) =0, \ \ \mbox{with}\  \ v(0)=v_0.\end{equation}

The Kato-Lai existence result for abstract evolution equations is as follows.

\begin{thm}Let $\{V, H, X\}$ be an admissible triplet. Let $A$ be a sequentially weakly continuous map from
$[0, T]\times H$ into $X$ such that
$$\langle v, A(t,v)\rangle \geq -\beta \left(\|v\|_H^2\right)\ \ \ \mbox{for}\ \ t\in [0,T], \ v\in V,$$
where $\beta(r)\geq 0 $ is a continuous nondecreasing function of $r\geq 0$.  For any $v_0\in H$, consider the ODE
$\gamma^\prime (t) = \beta(\gamma(t)), \ \gamma(0)=\|v_0\|_{H}^2,$. Suppose the maximal time of existence of $\gamma$ is $T_0$, then for (\ref{KL1}),
there exists a solution $v$ of (\ref{KL1}) in the class
$$v\in C_w ([0, T_0]; H)\cap C_w^1([0, T_0]; X).$$
Moreover, one has
$$\|v(t)\|_{H}^2\leq \gamma(t), \ \ \ t\in [0,T_0],$$
\end{thm}

In fact, it was proved in \cite{RR} that the triplet $\{\tilde{X},H_s,V_s\}$
is admissible for any smooth boundary.

The following lemma gives a decomposition of $L^2(\mathbb{R}^3)$. In particular, the second part is
a replacement of Proposition 3.1 of \cite{RR}.
\begin{lem} Let $G_1^2=\{u\in L^2(\mathbb{R}^3):\ u=\nabla q_1,\,\,  q_1\in L^1_{\rm loc}(\mathbb{R}^3) \}$, and
\begin{eqnarray*}
G_2^2=\left\{u\in L^2(\mathbb{R}^3):\  {\rm div}~u=0\,\, \mbox{in}
\,\,\mathbb{R}^3,\,\,u=\nabla q_2\  \mbox{in}
\ \Omega, q_2\in L^1_{\rm loc}(\Omega), \ \  u=\phi\,\, \mbox{in}\ \mathcal{O},\right. \\
\left. \int_{\mathcal{O}} \phi dy = -\int_{\partial \Omega} q_2
\vec{n} \ \dif\sigma,\ \ \mbox{and}\,\,\int_{\mathcal{O}}
 \phi\times y \dif y=-\int_{\partial\Omega}q_2\vec{n}\times y\,\dif\sigma\right\}.
\end{eqnarray*}
Then (1) $\tilde{X}_*,G_1^2$ and $G_2^2$ are mutually orthogonal in the sense of the standard $L^2$-inner product and
$$L^2(\mathbb{R}^3)=\tilde{X}_*\oplus G_1^2\oplus G_2^2.$$
It means that for every $u\in L^2(\mathbb{R}^3)$,
\begin{equation}u(y) =
\left\{\begin{array}{l} u_1+ \nabla q_1 + \nabla q_2,\ \ y\in \Omega\\
u_1 + \nabla q_1 + \phi,\ \ \ \ \ y\in \mathcal{O}\end{array}\right.
\in \tilde{X}_*\oplus G_1^2\oplus G_2^2. \label{3.1}\end{equation}
Suppose $u_1=l_{u_1}+\omega_{u_1}\times y$ in $\mathcal{O}$, then
\begin{equation}|l_{u_1}|+|\omega_{u_1}|\le C\|u_1\|_{L^2(\mathbb{R}^3)}
\leq C\|u\|_{L^2(\mathbb{R}^3)} \leq C\|u\|_{\tilde{X}}.\label{3.2}\end{equation}

(2)Denote the projector which maps $L^2(\mathbb{R}^3)$
to $\tilde{X}_*$ by $\mathbb{P}$. In fact, $\mathbb{P}$ maps $H_s$ into $H_s$
continuously for any $s\ge 0.$
\end{lem}

\begin{proof} The orthogonal decomposition in
(1) has been proved in \cite{DR}. And the estimate (\ref{3.2}) is partly derived in \cite{DR} and partly due
to the fact that $\|\cdot\|_{L^2(\mathbb{R}^3)}$ and $\|\cdot \|_{\tilde{X}}$ are equivalent norms.  Now we verify that (2) holds. For
every $u\in H_s(s\ge 0),$suppose that $u=u_1 + \nabla q_1 + \nabla q_2 = u_1 + \nabla P$ over $\Omega$,  then it suffices to prove that
$$\|\nabla P\|_{H^s(\Omega)}\leq C\|u\|_{H_s},$$
with some constant $C$ independent of $u$.

In fact, from the formula (\ref{3.1}), $P$ satisfies the following equations:
\begin{equation}\left\{\begin{aligned}
&\Delta P={\rm div}~u, \, \,& & \text{in}\,\, \Omega,\\
&\frac{\partial P }{\partial \vec{n}}=u\cdot
\vec{n}-(l_{u_1}+\omega_{u_1}\times y)\cdot \vec{n},
 \,\, && \text{on}\,\,\partial\Omega.\\
\end{aligned}\right.\label{3.3}
\end{equation}

Let \begin{equation*}
\varphi=\nabla\times\left[\frac{1}{2}\xi(y)\cdot \left(l_{u_1}\times
y-\omega_{u_1}|y|^2\right)\right],
\end{equation*}
where $\xi$ is a cut-off function defined in section 2. Clearly,
${\rm div}~\varphi=0$ in $\Omega$ and $\varphi \cdot
\vec{n}=(l_{u_1}+\omega_{u_1}\times y)\cdot \vec{n}$ on
$\partial\Omega$. Therefore, (\ref{3.3}) can be rewritten,
\begin{equation}\left\{\begin{aligned}
&\Delta P={\rm div}~(u-\varphi), \, \,& & \text{in}\,\, \Omega,\\
&\frac{\partial P }{\partial \vec{n}}=(u-\varphi)\cdot
\vec{n}, \,\, && \text{on}\,\,\partial\Omega.
\end{aligned}\right.\label{3.4}
\end{equation}

The solution to the system (\ref{3.4}) is closely related to the
Helmholtz-Weyl decomposition. As proved in \cite{G},
\begin{equation*} \begin{aligned}
\left\|\nabla P \right\|_{H^s(\Omega)}&\leq
 C\|u-\varphi\|_{H^s(\Omega)}\\
 & \leq C\left(\|u\|_{H^s(\Omega)}+ \|\varphi\|_{H^s(\Omega)}\right)\\
 & \leq C\left(\|u\|_{H^s(\Omega)} + |l_{u_1}|+ |\omega_{u_1}|\right)\\
 & \leq C\|u\|_{H_s},\end{aligned}\label{3.5}
\end{equation*}
which completes the proof of Lemma 3.3.
\end{proof}

The following lemma is to give the bounds of the terms which
appear in the system (\ref{2.16})-(\ref{2.22}). Before that we'd like to give another description of
relationship between $(l, \omega)$ and $(L, R)$, which is different from (\ref{2.14})-(\ref{2.14-add}).
Suppose $(L, R)$ is given, we want to determine $(l, \omega)$ and
then define the other coefficients in (\ref{2.16})-(\ref{2.22}).

Since $$\frac{dQ(t)}{dt}= A(\omega(t))Q(t), $$ multiplying by $Q^T(t)$, then
$$-\frac{dQ^T(t)}{dt}Q(t) = A(R(t)).$$
It gives that
$$\left\{\begin{array}{l}\frac{dQ(t)}{dt}= Q(t) A(R(t)),
 \\
Q(t)=Id.
  \end{array}\right.
$$
Now we see that if $R(t)$ is given, $Q(t)$ can be determined. Then $l(t), \omega(t)$ are determined,
\begin{equation}\label{3.5-add} l(t) = Q(t)L(t),\ \ \ A(\omega(t)) = Q(t) A(R(t))Q^T(t).
\end{equation}
Then $\Lambda, X, Y, g_{ij}, g^{ij}, \Gamma, Mv, Nv$ can be determined as in section 2.

Moreover, if $L(t),\ R(t)\in C[0, T]$, then $l(t), \ \omega(t)\in C[0,T]$, with the estimate
\begin{equation}\label{3.5-add-1}|l(t)| + |\omega(t)|\leq C(T)\left(\|L\|_{L^\infty(0,T)}+ \|R\|_{L^\infty(0, T)}\right).
 \end{equation}
Suppose that $(l^1(t), \omega^1(t))$ and $(l^2(t), \omega^2(t))$ are determined
by $(L^1(t), R^1(t))$ and $(L^2(t), R^2(t))$ respectively
in the above way, then
\begin{equation}\label{3.5-add-2}\begin{array}{ll}
&  \|l_1 -l_2\|_{L^\infty(0, T)}+\|\omega_1 -\omega_2\|_{L^\infty(0, T)}\\
\leq  & C(T)(1+ \|R_1(t)\|_{L^\infty(0, T)}+
\|R_2(t)\|_{L^\infty(0, T)})\cdot (\|L_1 - L_2\|_{L^\infty(0, T)}
+\|R_1 -R_2\|_{L^\infty(0, T)}).
\end{array}\end{equation}

\begin{lem}Assume that $v$ is a function in $L^\infty (0, T; \tilde{X}_*)$ and $s$ is a nonnegative integer.
 Suppose there exists $M_*>0$,  such that
 $\|v\|_{L^{\infty}(0,T; \tilde{X})}\le M_*$. Let
$$L(t) = l_{v(t)},\ \ \ \ \ R(t)=\omega_{v(t)}.$$
Suppose $l(t), \omega(t)$ is given by $L(t), R(t)$ as in (\ref{3.5-add}) and
 $\Lambda, X, Y, g_{ij}, g^{ij}, \Gamma$ are defined as
  in section 2.  Then for every $t\in [0, T]$, the following estimates
 hold:
 \begin{alignat}{12}
 &\|J_X(\cdot,t)\|_{W^{s,\infty}(\mathbb{R}^3)}\le C(s,T,M_*),\,\,
 \|J_Y(X(\cdot,t),t)\|_{W^{s,\infty}(\mathbb{R}^3)}\le C(s, T,M_*),\label{3.6}\\
 &\|\Lambda(X(\cdot,t),t)\|_{W^{s,\infty}(\mathbb{R}^3)}\le C(s, T,M_*),\,\,
 \|g_{ij}(\cdot,t)\|_{W^{s,\infty}(\mathbb{R}^3)}\le C(s, T,M_*),\label{3.7}\\
 &  \|g^{ij}(\cdot ,t)\|_{W^{s,\infty}(\mathbb{R}^3)}\le C(s, T,M_*),
 \,\, \|G^{-1}(\cdot, t)\|_{W^{s, \infty}(\mathbb{R}^3)}\leq C(s, T,
 M_*),\label{3.8-add}
 \\ & \|\Gamma(\cdot,t)\|_{W^{s,\infty}(\mathbb{R}^3)}\le
 C(s, T,M_*),\label{3.8}
 \end{alignat}
 where $G^{-1}$ is the inverse of $G$.
 \end{lem}
 \begin{proof}
For every $j=1, 2$ or $3$, let $z(y,t) = \frac{\partial X}{\p y_j},$
 \begin{equation*}
\left\{\begin{aligned}
&\frac{\p z(y,t)}{\p t}=\frac{\p\Lambda}{\p x}(X(y,t), t)\cdot z(y,t),\\
&z(y, 0)=\vec{e}_j ,
\end{aligned}\right.
\end{equation*}
where $\vec{e}_j$ is the $j$-th vector of the basis of
$\mathbb{R}^3.$  Then
\begin{equation}
z(y,t)=\vec{e}_j+\int_0^t\frac{\p\Lambda}{\p x}(X(y, t), t)\cdot
z(y,s)\dif s.\label{3.9}
\end{equation}
It follows from Gronwall's inequality that $|z(y,t)|\leq C(T,M_*)$.

Since $\det (J_X)=1$, then $ J_Y=(J_X^{ij})$, where $J_X^{ij}$ is
the cofactor of $J_X$. Hence $$|J_Y(X(\cdot,t),t)|\le C(T, M_*).$$
Furthermore,
$$|G(\cdot,t)|\le
C(T,M_*),\ \ \ |G^{-1}(\cdot,t)|\le C(T,M_*).$$

Denote $D^\beta=\frac{\p^\beta}{\p y^\beta}$. From (\ref{3.9}), one
can deduce that
$$D^\alpha z=\int_0^t \sum_{\beta \le \alpha}{\alpha \choose\beta} D^\beta\left(\frac{\p\Lambda}
{\p x}\right)D^{\alpha-\beta}z \dif s,\,\, \ |\alpha|\leq s.$$
Following the preceding process, one can get the estimates
(\ref{3.6})-(\ref{3.8}).
 \end{proof}

Next lemma is about the Lipschitz continuity of the coefficients
with respect to $v$.
\begin{lem} Assume that the assumptions of
Lemma 3.4 hold for $v^i,\,\, i=1,2.$ Let $L^i(t) = l_{v^i(t)}$ and $R^i(t)= \omega_{v^i(t)}.$ Define
$Q^i(t), l^i(t), \omega^i(t)$ and other terms correspondingly.
Let $L(t)=L^1(t)- L^2(t),\ \ R(t)=
R^1(t)-R^2(t), \,\, X=X^1-X^2,\,\,
\Lambda(y,t)= \Lambda(X^1(y, t), t) - \Lambda(X^2(y, t), t), \,\,
G=(g^{ij})=(g^{ij,1}-g^{ij,2}),\,\,g_{ij}=g_{ij}^1-g_{ij}^2, $\,\
$G^{-1} = (G^1)^{-1} - (G^2)^{-1}$, and\,\,
$\Gamma^j_{m,k}=\Gamma^{j,1}_{m,k}-\Gamma^{j,2}_{m,k}.$ Then for
every $t\in [0,T]$,
\begin{equation}\label{3.10}
\|X(\cdot,  t)\|_{W^{s, \infty}(\mathbb{R}^3)}\leq C(s, T,
M_*)\left(\|L\|_{L^\infty(0, T)} + \|R\|_{L^\infty(0,
T)}\right),
\end{equation}
\begin{equation}
\label{3.11} \|\Lambda(\cdot, t)\|_{W^{s, \infty}(\mathbb{R}^3)}\leq
C(s, T, M_*) \left(\|L\|_{L^\infty(0, T)} + \|R\|_{L^\infty(0,
T)}\right),
\end{equation}
\begin{equation}
\label{3.12} \|g_{ij}(\cdot, t)\|_{W^{s, \infty}(\mathbb{R}^3)}\leq
C(s, T, M_*)\left(\|L\|_{L^\infty(0, T)} + \|R\|_{L^\infty(0,
T)}\right),
\end{equation}
\begin{equation}\label{3.13}
\|g^{ij}(\cdot, t)\|_{W^{s, \infty}(\mathbb{R}^3)} \leq C(s, T, M_*)
\left(\|L\|_{L^\infty(0, T)} + \|R\|_{L^\infty(0, T)}\right),
\end{equation}
\begin{equation}\label{3.14}
\|G^{-1}(\cdot, t)\|_{W^{s, \infty}(\mathbb{R}^3)}\leq C(s, T,
M_*)\left(\|L\|_{L^\infty(0, T)} + \|R\|_{L^\infty(0,
T)}\right),
\end{equation}
\begin{equation}\label{3.15}
\|\Gamma(\cdot, t)\|_{W^{s, \infty}(\mathbb{R}^3)}\leq C(s, T,
M_*)\left(\|L\|_{L^\infty(0, T)} + \|R\|_{L^\infty(0,
T)}\right).
\end{equation}
\end{lem}

\begin{proof}
Since $$\left\{\begin{array}{l} \frac{\partial X(y, t)}{\partial t}=
\Lambda(X^1(y,t), t)- \Lambda(X^2(y, t), t),\\[2mm]
X(y, 0) = 0,
\end{array}\right.$$
Simple calculation and the estimates in Lemma 3.4 induce that
$$\begin{array}{ll}\left|\Lambda(X^1(y,t), t) - \Lambda(X^2(y, t), t)\right| &
\leq C(T, M_*) |X^1(y,t) - X^2(y, t)| \\[2mm]
& \ \ \ \  + C(T, M_*)(|L(t)| + |R(t)|).
\end{array}$$
Therefore,
$$\|X\|_{C([0, T]; L^\infty(\mathbb{R}^3))}\leq
C(T, M_*) \left(\|L\|_{L^\infty(0, T)} + \|R\|_{L^\infty(0,
T)}\right).$$

$$\begin{array}{ll} \|\Lambda(y, t)\|_{L^\infty(\mathbb{R}^3)}
& \leq C(T, M_*) \|X^1(y, t)- X^2(y,
t)\|_{L^\infty(\mathbb{R}^3)}\\[2mm]
& \leq C(T, M_*)\left(\|L\|_{L^\infty(0, T)} +
\|R\|_{L^\infty(0, T)}\right). \end{array}$$

Other estimates can be derived similarly.
\end{proof}

\section{A Priori $H^s$-estimates of $\nabla q$}
In the following text, $s>\frac52$. Given a function $v \in
C_w(0, T; H_s\cap \tilde{X}_*)$, which satisfies that
$\|v\|_{L^\infty(0, T; H_{s})}\leq M_0.$ Suppose $v$ is a solution to (\ref{2.16})-(\ref{2.23}), taking the divergence of Eq. (\ref{2.16}), then
$q$ satisfies the following system,
\begin{equation}\left\{\begin{aligned}
&{\rm div}\left(\sum_{j=1}^{3}g^{ij}\frac{\partial q}{\partial
y_j}\right)
=-\di (M v+N v), \, \,& & \text{in}\,\, \Omega,\\
&\sum_{i,j=1}^3g^{ij}\frac{\partial q }{\partial
y_j}n_i+\left(\frac{1}{m} \int_{\partial\Omega} q
\vec{n}\dif\sigma\right)\cdot \vec{n}
+\left(\bar{J}^{-1}\int_{\partial\Omega}y\times
q\vec{n}\dif\sigma
\right)\times y\cdot \vec{n}\\
&=-(Mv+Nv)\cdot \vec{n}+(\omega_v\times l_v)\cdot \vec{n}
-\left[\bar{J}^{-1}(\bar{J}\omega_v\times\omega_v)\right]\times
y\cdot \vec{n}, \,\, && \text{on}\,\,\partial\Omega.
\end{aligned}\right.\label{4.1}
\end{equation}
Here $Q$, $l$, $\omega$, $g^{ij}$, $Mv$ and $N v$ are given as in (\ref{3.5-add}), (\ref{2.14})-(\ref{2.15}) and (\ref{2.23}), replacing
$L, R$ by $l_v, \omega_v$.

In this section, we will give the $H^s$-estimates of $\nabla q$, for every fixed time $t\in [0, T]$. For simplicity of writing, we omit $t$. Here is the main result of
this section.

\begin{prop}$$\|\nabla q\|_{H^s(\Omega)} \leq C(T, M_0) (1+ \|v\|_{H_s}), $$
with some constant $C$ depending on $T$ and $M_0$.
\end{prop}
\begin{rem}
The system $q$ satisfies is almost a classical elliptic problem with Neumann boundary condition. During the proof of Proposition 4.1, the main idea is to use the Lax-Milgram theorem
to get the existence of weak solution, and the standard high-order regularity estimates for exterior elliptic problems.
\end{rem}

\begin{proof}For every fixed $t\in [0, T]$, the matrix  $ G=
(g^{ij})=J_{Y}J_{Y}^T$, so $G $ is positive definite. Denote
$\lambda_i(y,t)>0,(i=1,2,3)$  the eigenvalues of the matrix
$(g^{ij})$. Since $\det(g^{ij})=1$,
  thus $\prod_{i=1}^{3}\lambda_i=1$ and $\sum\limits_{i=1}^3 \lambda_i=\sum\limits_{i=1}^3 g^{ii}>0$.
  Let $\gamma_0=\sup\limits_{y\in\mathbb{R}^3}|g^{ii}|$,
  then we have $3\gamma_0\ge\lambda_i\ge\frac{1}{(3\gamma_0)^{2}}$ for every
  $i=1,2,3.$ By virtue of Lemma 3.4, there exist constants $C_1(T, M_0)$ and
  $C_2(T, M_0)$,
  $$C_1(T, M_0) \leq |\lambda_i| \leq C_2(T, M_0), \ \ \ i=1,2,3.$$

  Next, we shall use the Lax-Milgram theorem to prove the existence of the solution of (\ref{4.1}),
  and then give the $H^s$-estimate of this solution.

 Set a bilinear form $B$ and a linear functional $F$ on $\dot{D}^{1,2}(\Omega)$ as
 follows, for every $\eta,q\in \dot{D}^{1,2}(\Omega),$
\begin{equation*}\begin{aligned}
  B(q,\eta)&=\sum_{i=1}^3\left(\sum_{j=1}^3 g^{ij}\frac{\partial q}{\partial y_j},
  \frac{\partial \eta}{\partial y_i}\right)_{L^2(\Omega)}+\frac{1}{m}
  \left(\int_{\partial\Omega} q \vec{n}\dif\sigma
  \right)\cdot \left(\int_{\partial\Omega} \xi \vec{n}\dif\sigma\right)\\
  &\ \ + \left(\bar{J}^{-1}\int_{\partial\Omega}y\times q\vec{n}
  \dif\sigma\right)\cdot \left(\int_{\partial\Omega}y\times \eta \vec{n}\dif\sigma\right),\\
  F(\eta)&=-\int_{\Omega}(Mv+Nv)\cdot\nabla\eta dy+ \int_{\partial\Omega}
  (\omega_v\times l_v)\cdot\eta \vec{n}
  \dif\sigma\\
  &\ \ -\int_{\partial\Omega}\left[\bar{J}^{-1}(\bar{J}\omega_v
  \times\omega_v)\right]\times y\cdot \eta \vec{n}\dif\sigma.
  \end{aligned}
\end{equation*}

Note that
\begin{equation}\begin{aligned}
  B(q,q)&=\sum_{i=1}^3\left(\sum_{j=1}^3 g^{ij}\frac{\partial q}{\partial y_j},\frac{\partial q}
  {\partial y_i}\right)_{L^2(\Omega)}+\frac{1}{m}\left(\int_{\partial\Omega} q \vec{n}\dif\sigma\right)^2
  +\bar{J}w(q)\cdot w(q)\\
  &\ge C_1(T, M_0)\|\nabla q\|^2_{L^2(\Omega)}+ \frac{1}{m}\left(\int_{\partial\Omega} q \vec{n}\dif\sigma\right)^2
  +\bar{J}w(q)\cdot
  w(q),
    \end{aligned}\label{4.2}
\end{equation}
where $w(q)=\bar{J}^{-1}\int_{\partial\Omega}y\times q \vec{n}\dif\sigma$.

Since $\bar{J}$ is a positive definite matrix, then there exists some
constant $a>0$ such that
$$a^{-1}|w(q)|^2\le\bar{J}w(q)\cdot w(q)\leq a|w(q)|^2.$$
Combining the above inequality and (\ref{4.2}), one gets that $B$ is
coercive.

On the other hand, $$\left|\sum_{i=1}^3\left(\sum_{j=1}^3
g^{ij}\frac{\partial q}{\partial y_j},\frac{\partial \eta}{\partial
y_i}\right)_{L^2(\Omega)}\right|\leq
\|G\|_{L^{\infty}(\mathbb{R}^3)}\|\nabla
q\|_{L^2(\Omega)}\|\nabla\eta\|_{L^2(\Omega)}.$$ Then along the line
of the proposition 3.3.1 in \cite{WX2}, one can easily verify that
the bilinear form $B$ is bounded.

Now we turn to the functional $F$.
\begin{equation}\begin{aligned}
  &\ \ \ \left|-\int_{\Omega}(Mv+Nv)\cdot\nabla\eta  dx\right| \\
  &\leq \|Mv+Nv\|_{L^2(\Omega)} \cdot \|\nabla\eta\|_{L^2(\Omega)}\\
  &\leq C\left( \|\Lambda\|_{W^{1,\infty}(\Omega)}+\|J_Y\|_{L^{\infty}(\Omega)}
  +\|J_{X}\|_{L^{\infty}(\Omega)}+\|v\|_{L^{\infty}(\Omega)}\right.\\&\ \ \ \ \ \ \ \left. +
  \|\Gamma\|_{L^\infty(\Omega)}\right)
\cdot\|v\|_{H^1(\Omega)}\cdot\|\nabla\eta\|_{L^2(\Omega)}
  \\ & \leq C(T, M_0)\|\nabla \eta\|_{L^2(\Omega)},
  \end{aligned}\label{4.3}
\end{equation}

Choosing some $\eta\in D^{1,2}(\Omega)$ such that
$\int_{\Omega_r}\eta dy =0$, for some large $r$.
\begin{equation}\begin{aligned}
&\ \ \ \left|\int_{\p \Omega}\left\{\omega_v\times
l_v-[\bar{J}^{-1}(\bar{J}\omega_v\times\omega_v)]\times y\right\}
\cdot\eta \vec{n} \dif\sigma \right| \\& \leq C (|\omega_v||l_v|+|\omega_v|^2)\|\eta\|_{L^2(\partial \Omega)}\\
&\le
C(\Omega)(|\omega_v||l_v|+|\omega_v|^2)\|\eta\|_{H^1(\Omega_r)}\\
&\leq C(T, M_0, \Omega) \|\nabla \eta\|_{L^2(\Omega)},
\end{aligned}\label{4.5}
\end{equation}
where the last inequality is given by the Poincar\'{e}'s inequality.

From  the above estimates, it follows that $F$ is bounded. According to
Lax-Milgram Theorem, there exists a unique
$q\in\dot{D}^{1,2}(\Omega)$ such that
$$B(q,\eta)=F(\eta),\,\, \forall\  \eta\in\dot{D}^{1,2}(\Omega)$$
with
\begin{equation}
\|\nabla q\|_{L^2(\Omega)}\le C(T, M_0).\label{4.6}
\end{equation}

Let $$L_1 = \frac{1}{m} \int_{\p \Omega} q\vec{n} d\sigma, \ \ \
w=\bar{J}^{-1}\int_{\partial \Omega} y\times q\vec{n} d\sigma.$$
Then
\begin{equation}
|L_1|\leq C\|\nabla q\|_{L^2(\Omega)}\leq C(T, M_0),\ \ \ |w|
\leq C\|\nabla q\|_{L^2(\Omega)}\leq C(T, M_0).
\label{4.7}
\end{equation}

 Now we go to the $H^s$-estimate of $\nabla q$. Similar to \cite{RR},
 the method is a standard regularity estimate for an exterior problem of elliptic equations.
Consider the Neumann system
which is equivalent to (\ref{4.1}),
\begin{equation}\left\{\begin{aligned}
&{\rm div}\left(\sum_{j=1}^{3}g^{ij}\frac{\partial q}{\partial
y_j}\right)=-\di (M v+N v),
 \, \,& & \text{in}\,\, \Omega,\\
&\sum_{i,j=1}^3g^{ij}\frac{\partial q }{\partial y_j}n_i
=-(Mv+Nv)\cdot \vec{n} - (\omega_v\times l_v) \cdot \vec{n} \\
&\ \ \ +\left[\bar{J}^{-1}(\bar{J}\omega_v\times\omega_v)\right]\times
y\cdot \vec{n}-L_1\cdot \vec{n}- (w\times y)\cdot \vec{n},
\,\, && \text{on}\,\,\partial\Omega.
\end{aligned}\right.\label{4.8}
\end{equation}

To estimate $\|\nabla q\|_{H^s(\Omega)}$, the key is to estimate the
terms $\|\di (Mv+Nv)\|_{H^{s-1}(\Omega)}$ and $\|-(Mv+Nv)\cdot
\vec{n}\|_{H^{s-\frac{1}{2}}(\p \Omega)}$.

\begin{alignat*}{12}
& \left\|\di\left(\left(\frac{\p Y}{\p t}+v\right)\cdot \nabla
v\right)\right\|_{H^{s-1}(\Omega)}
\\ = & \left\|\sum_{i,j=1}^3\frac{\p}{\p y_i} \left(\frac{\p Y_j}{\p t}+ v_j\right)\frac{\p v_i}{\p y_j}\right\|_{H^{s-1}(\Omega)}\\
\leq & \left\|\sum_{i,j,k=1}^3\frac{\p^2 Y_j}{\p t\p x_k}\frac{\p
X_k}{\p y_i}\frac{\p v_i}{\p y_j}\right\|_{H^{s-1}(\Omega)}
+\left\|\sum_{i,j=1}^3\frac{\p v_j}{\p y_i}\frac{\p v_i}{\p y_j}\right\|_{H^{s-1}(\Omega)}\\
\leq &
C\left(\left\|\Lambda\right\|_{W^{s,\infty}(\mathbb{R}^3)}\left\|J_{X}\right\|_{W^{s-1,\infty}
(\mathbb{R}^3)}\|\nabla v\|_{H^{s-1}(\Omega)}+ \|\nabla
v\|_{H^{s-1}(\Omega)}\|v\|_{H^{s_0}(\Omega)}\right),
\end{alignat*}
and
\begin{alignat*}{12}
& \left\|\sum_{i=1}^3 \frac{\p}{\p y_i
}\left.\left(\sum_{j,k=1}^3 \left\{\Gamma_{j,k}^i\frac{\p Y_k}{\p
t}+\frac{\p Y_k}{\p x_k}\frac{\p^2 X_k}{\p t\p y}\right\}v_j
+\sum_{j,k=1}^3 \Gamma_{j,k}^i v_j v_k\right) \right.\right\|_{H^{s-1}(\Omega)}\\
\leq &
C\left[\left(\|\Gamma\|_{W^{s,\infty}(\mathbb{R}^3)}\|\Lambda\|_{W^{s,\infty}(\mathbb{R}^3)}
+\|J_Y\|_{W^{s,\infty}(\mathbb{R}^3)}\|\Lambda\|_{W^{s,\infty}(\mathbb{R}^3)}\right)\cdot \left\|v\right\|_{H^{s}(\Omega)}\right.\\
&\ \ \
+\left.\|\Gamma\|_{W^{s,\infty}(\mathbb{R}^3)}\|v\|_{H^s(\Omega)}\|v\|_{H^{s}(\Omega)}\right].
\end{alignat*}

Hence,
\begin{equation}
\|\di (Mv+Nv)\|_{H^{s-1}(\Omega)}\leq C(T, M_0)\|v\|_{H_{s}}.\label{4.9}
\end{equation}

Denote $$I =\left(\sum_{j,k=1}^3\left\{\Gamma_{j,k}^i\frac{\p Y_k}{\p t}+\frac{\p
Y_k}{\p x_k}\frac{\p^2 X_k}{\p t\p
y}\right\}v_j+\sum_{j,k=1}^3\Gamma_{j,k}^i v_j v_k \right)_i.$$ Then
$$-(Mv + Nv)\cdot \vec{n} = \left(\frac{\p Y}{\p t} + v \right)\cdot \nabla v
\cdot \vec{n} + I\cdot \vec{n}.$$

Note that there is no derivative of $v$ in the term $I$, so it is easy to handle.
\begin{equation}\label{4.10}\begin{aligned}  \ \ \ \|I\cdot \vec{n}\|_{H^{s-\frac{1}{2}}(\p \Omega)}
& \le C\|I\|_{H^s(\Omega)}\\ & \le C \left[
\left(\|\Gamma\|_{W^{s,\infty}(\mathbb{R}^3)}\|\Lambda\|_{W^{s,\infty}(\mathbb{R}^3)}
+\|J_Y\|_{W^{s,\infty}(\mathbb{R}^3)}\|\Lambda\|_{W^{s,\infty}(\mathbb{R}^3)}\right)\right. \\
&\ \ \ \ \ \left. \cdot \left\|v\right\|_{H^{s}(\Omega)}+
\|\Gamma\|_{W^{s,\infty}(\mathbb{R}^3)}\left\|v\right\|_{H^{s}(\Omega)}\left\|v\right\|_{H^{s}(\Omega)}\right]\\
& \leq C(T, M_0)\|v\|_{H_{s}}.
\end{aligned}
\end{equation}

To estimate $ \left\|(\frac{\p Y}{\p t}+v)\cdot\nabla v\cdot
\vec{n}\right\|_{H^{s-\frac{1}{2}}(\p \Omega)}$, we proceed as in \cite{BB, RR}. Note $(\frac{\p Y}{\p t}+v)\cdot
\vec{n}=0$ on $\p \Omega$, then
$$\begin{array}{ll} &\displaystyle \left(\frac{\partial Y}{\partial t} + v \right)\cdot \nabla v \cdot \vec{n} \\[3mm]
=& \displaystyle\left( \frac{\partial Y}{\partial t} +v \right) \cdot \nabla \left(\frac{\partial Y}{\partial t}+v\right)\cdot \vec{n}
- \left(\frac{\partial Y}{\partial t} + v\right)\cdot \nabla \left(\frac{\partial Y}{\partial t}\right)\cdot \vec{n}\\[3mm]
= &\displaystyle - \sum_{i, j=1}^3 \left(\frac{\partial Y}{\partial t}\right)_i \left(\frac{\partial Y}{\partial t}+ v\right)_j \partial_i n_j
- \left(\frac{\partial Y}{\partial t} + v\right)\cdot \nabla \left(\frac{\partial Y}{\partial t}\right)\cdot \vec{n}.
\end{array} $$

Hence, \begin{equation} \begin{aligned}\left\|\left(\frac{\p Y}{\p t}+v\right)\cdot\nabla v\cdot
\vec{n}\right\|_{H^{s-\frac{1}{2}}(\p \Omega)} \le
C\left(\|\Lambda\|^2_{W^{s,\infty}(\mathbb{R}^3)}+1\right)\left(1+\|v\|_{H^{s}(\Omega)}\right).\end{aligned}
\label{4.11}
\end{equation}

Combining (\ref{4.10}) and (\ref{4.11}), one has
\begin{equation}
\|-(Mv+Nv)\cdot \vec{n}\|_{H^{s-\frac{1}{2}}(\p \Omega)}\leq C(T,
M_0)(1+\|v\|_{H_{s}}).\label{4.12}
\end{equation}

The other terms can be estimated as follows:
\begin{equation}
\|\left[\bar{J}^{-1}(\bar{J}\omega\times\omega)\right]\times y\cdot
\vec{n}\|_{H^{s-\frac{1}{2}}(\p\Omega)}\leq
 C(\Omega,R)|\omega|^2 \leq C(T, M_0), \label{4.13}\end{equation}

\begin{equation}\|L_1\cdot \vec{n}\|_{H^{s-\frac{1}{2}}(\p\Omega)}\le C
|L_1|\leq C(T, M_0), \label{4.14}
\end{equation}

\begin{equation} \|(\omega_v\times l_v)\cdot \vec{n}\|_{H^{s-\frac{1}{2}}(\p\Omega)}\leq
C(\Omega,R) |\omega||l| \leq C(T, M_0),\label{4.15}\end{equation}

\begin{equation}\|(w\times y)\cdot
\vec{n}\|_{H^{s-\frac{1}{2}}(\p\Omega)}\le C(\Omega,R)|w| \leq
C(T, M_0).\label{4.16}
\end{equation}

Choose some $r>0$ such that ${\rm supp}(\xi)\subset
B_{\frac{r}{2}}$, and a cut-off function $\xi_1$,
\begin{equation*}\begin{aligned}
\xi_1(y)=\begin{cases} 1, & \ \mbox{if}\ |y|\le 2r,\\0, &\
\mbox{if}\ |y|\ge 3r.
\end {cases}
\end{aligned}
\end{equation*}
Hence, $p_1=\xi_1q$ solves the following equation
\begin{equation}\left\{\begin{aligned}
{\rm div}\left(\sum_{j=1}^{3}g^{ij}\frac{\partial p_1}{\partial y_j}\right)&=-\xi_1\di (M v+N v)+\sum_{i,j=1}^3\frac{\p g^{ij}}{\p y_i}\frac{\p \xi_1}{\p y_j}q\\
& +\sum_{i,j=1}^3g^{ij}\left(\frac{\p\xi_1}{\p y_j}\frac{\p q}{\p y_i}+\frac{\p\xi_1}{\p y_i}\frac{\p q}{\p y_j}
+ \frac{\partial^2 \xi_1}{\partial y_i \partial y_j}q\right), \, \,& & \text{in}\,\, B_{4r}\setminus\mathcal{O},\\
\sum_{i,j=1}^3g^{ij}\frac{\partial p_1 }{\partial y_j}n_i
&=-(Mv+Nv)\cdot \vec{n}-\omega_v\times l_v\cdot \vec{n} \\
&+\left[\bar{J}^{-1}(\bar{J}\omega_v\times\omega_v)\right]
\times y\cdot \vec{n}-L_1\cdot \vec{n}-w\times y\cdot \vec{n}, \,\, && \text{on}\,\,\partial\mathcal{O},\\
\sum_{i,j=1}^3g^{ij}\frac{\partial p_1 }{\partial y_j}n_i &=0, \,\,
&& \text{on}\,\, \p B_{4r}.
\end{aligned}\right.\label{4.17}
\end{equation}

By virtue of the regularity theory for elliptic equations\cite{So}, for any $\beta\geq 1$,
\begin{equation}
\begin{aligned}
&\| p_1\|_{H^{\beta+1}(B_{4r}\setminus\mathcal{O})}\leq
h_1\left(\|G\|_{W^{s,\infty}(\mathbb{R}^3)}, (3\gamma_0)^2\right)
\left(\left\|\sum_{i,j=1}^3\frac{\p g^{ij}}{\p y_i}\frac{\p
\xi_1}{\p y_j}q
\right\|_{H^{\beta-1}(B_{4r}\setminus\mathcal{O})}\right.\\ &+\|\xi_1\di
(M v+N v)\|_{H^{\beta-1}(B_{4r}\setminus\mathcal{O})}+
\left\|\sum_{i,j=1}^3g^{ij} \left(\frac{\p\xi_1}{\p y_j}\frac{\p
q}{\p y_i}+\frac{\p\xi_1}{\p y_i} \frac{\p q}{\p
y_j}\right)\right\|_{H^{\beta-1}(B_{4r}\setminus\mathcal{O})}
\\
& + \left\|\sum_{i,j=1}^3 g^{ij} \frac{\partial^2 \xi}{\partial y_i \partial y_j}q\right\|_{H^{\beta -1}(B_{4r}\setminus \mathcal{O})}
+\|-(Mv+Nv)\cdot \vec{n}\|_{H^{\beta-\frac{1}{2}}(\p\Omega)}\\ & \left.
+\left\|\bar{J}^{-1}(\bar{J}\omega\times\omega)\cdot \vec{n}\right\|_{H^{\beta-\frac{1}{2}}(\p\Omega)}+\|(L_1+w\times y)\cdot
\vec{n}\|_{H^{\beta-\frac{1}{2}}(\p\Omega)}
+\|p_1\|_{L^2(B_{4r}\setminus\mathcal{O})}\right),
\end{aligned}\label{4.18}
\end{equation}
where $h_1(\cdot, \cdot)$ can be chosen an increasing function with
respect to both variables. In fact,
\begin{equation}
\begin{aligned}
&\ \ \ \left\|\sum_{i,j=1}^3\frac{\p g^{ij}}{\p y_i} \frac{\p
\xi_1}{\p y_j}q\right\|_{H^{\beta-1}(B_{4r}\setminus\mathcal{O})}+
\left\|\sum_{i,j=1}^3g^{ij}\left(\frac{\p\xi_1}{\p y_j} \frac{\p
q}{\p y_i}+\frac{\p\xi_1}{\p y_i}\frac{\p q}{\p y_j}\right)
\right\|_{H^{\beta-1}(B_{4r}\setminus\mathcal{O})}\\&\ \ \
+ \left\|\sum_{i,j=1}^3 g^{ij} \frac{\partial^2 \xi}{\partial y_i \partial y_j}q\right\|_{H^{\beta -1}(B_{4r}\setminus \mathcal{O})}
\\
&\le C\left( \|G\|_{W^{s,\infty}(\mathbb{R}^3)}\cdot
\|q\|_{H^{\beta-1}(B_{3r}\setminus B_{2r})} + \|G\|_{W^{\beta-1,
\infty}(\mathbb{R}^3)}\cdot \|\nabla q\|_{H^{\beta-1}(B_{3r}\setminus
B_{2r})}\right).
\end{aligned}\label{4.19}
\end{equation}

Combining the above estimates, one gets
\begin{equation}
\begin{aligned}
&\|q\|_{H^{\beta+1}(B_{2r}\setminus \mathcal{O})} \\ \leq & C(T, M_0, r)
\left( \|q\|_{H^{\beta-1}(B_{3r}\setminus B_{2r})} +
\|q\|_{L^2(B_{3r}\setminus \mathcal{O})}+ 1\right).
\end{aligned} \label{4.20}
\end{equation}

Choose some particular $q$ such that
$$\int_{B_{4r}\setminus \mathcal{O}} q(y) dy = 0.$$
It is reasonable, since $q$ is still a solution to (\ref{4.8}) if it is added by a constant. By Poincar\'{e}'s inequality,
\begin{equation}
\begin{aligned}
\|q\|_{H^{\beta+1}(B_{2r})} &\leq C(T, M_0,
r)\left(\|q\|_{H^\beta(B_{3r}\setminus B_{2r})} + \|\nabla
q\|_{L^2(B_{4r}\setminus \mathcal{O})} + 1 \right) \\
& \leq C(T, M_0, r) \left( \|q\|_{H^\beta(B_{3r}\setminus
B_{2r})} + 1\right).
\end{aligned}\label{4.21}
\end{equation}

It implies that high-order regularity of $q$ can be controlled by
the lower-order regularity. Therefore, using this method by choosing
appropriate $r$, we can get that for every $R>0$, such that
$\mathcal{O}\subseteq B_{\frac{R}{2}},$
\begin{equation}\label{4.22}
\|\nabla q\|_{H^s(\Omega_R)}\leq C(T, M_0, R)(1+\|v\|_{H_{s}}).
\end{equation}

Fix some $R$ big enough. Choose some smooth cut-off function
$\xi_2$, such that
 \begin{equation*}\begin{aligned}
 \xi_2(y)=\begin{cases} 0, &\ \ \ \mbox{if}\  |y|\le\frac{3}{4}R, \\1, &
 \ \ \ \mbox{if}\ |y|\ge R.
\end {cases}
\end{aligned}
\end{equation*}

 Since $g^{ij}=\delta_{ij}$ outside $B_{\frac{R}{2}}$, hence
  $$\di (G\cdot\nabla q)=\Delta q.$$
Let $p_2=\xi_2 q$, then \begin{equation}\label{4.23} \Delta
p_2=\xi_2(-\di (Mv+Nv))+ 2\nabla\xi_2\cdot \nabla q+\Delta\xi_2
q:=\tilde{f}. \end{equation} Therefore,
\begin{equation}\label{4.24}
\|\nabla p_2\|_{H^s(\mathbb{R}^3)} \leq
C\left(\|\tilde{f}\|_{H^{s-1}(\mathbb{R}^3)}+
\|\nabla p_2\|_{L^2(\mathbb{R}^3)}\right).
\end{equation}

$\tilde{f}$ is estimated as follows,
\begin{equation}
\begin{aligned}
&\ \ \ \ \|\tilde{f} \|_{H^{s-1}(\mathbb{R}^3)} \\
& \le C\left(\|\di (Mv+Nv)\|_{H^{s-1}(\Omega)}+\|\nabla
q\|_{H^{s-1}(\frac{R}{2}\le|y|\le R)}
+\|q\|_{H^{s-1}(\frac{R}{2}\le|y|\le R)}\right)\\
&\leq C(T, M_0)(1+\|v\|_{H_{s}}).
\end{aligned}\label{4.25}
\end{equation}

Hence \begin{equation}\label{4.26} \|\nabla
q\|_{H^s(\mathbb{R}^3\setminus B_{R})} \leq C(T, M_0)(1+\|v\|_{H_{s}}).
\end{equation}

(\ref{4.22}) and (\ref{4.26}) give that
\begin{equation}
\label{4.27} \|\nabla q\|_{H^s(\Omega)}\leq C(T, M_0)(1+\|v\|_{H_{s}}).
\end{equation}
It completes the proof of Proposition 4.1.
\end{proof}

\section{Construction of approximate solutions}
In this section, we will construct a sequence of approximate
solutions. Similar to \cite{RR}, the main idea is the Kato-Lai theory.
Suppose for every $n\geq 0$, there is a function $v^n(t)\in C_w([0, T]; H_s)$. Denote $L^n(t)=l_{\P v^n(t)}$,
$R^n(t)=\omega_{\P v^n(t)}.$ Then $L^n(t), R^n(t)\in C[0, T]$. Solving the following initial value
problem
\begin{equation}\left\{\begin{aligned}
\frac{\dif Q^n(t)}{\dif t} &=Q^n(t) A(R^n(t)),\\
Q^n(0)&=Id.
\end{aligned}\right.\label{5.1}
\end{equation}
One can get a solution $Q^n(t)$.

Define
$$l^n(t) =  Q^n(t)L^n(t),\ \ \ A(\omega^n(t))= Q^n(t) A(R^n(t)) [Q^n(t)]^T,$$
$$\psi^n=\xi\left((Q^{n})^{T}(t)(x-h^n(t))\right), \,\, h^n=\int_0^tQ^{n}(s) l^n(s)\dif s ,$$ and
$$V^n=l^n(t)+\omega^n(t)\times(x-h^n(t)),$$
$$W^n=l^n(t)\times(x-h^n(t))+\frac{|x-h^n(t)|^2}{2}\omega^n(t),$$
where $\xi$ is a cut-off function given in section 2.

Let $\Lambda^n=\psi^n V^n+\nabla\psi^n W^n.$ Hence one can define
$X^n(\cdot,t)$, $\Lambda^n(X^n(\cdot, t),t)$, $g^{ij,n}$,
$g^n_{ij}$, $\Gamma^{i,n}_{jk}$, $M^n, N^n$ given in (\ref{2.15}) and (\ref{2.23}). Suppose
that $q^n$ is the solution to the following system,
\begin{equation}\left\{\begin{aligned}
&{\rm div}\left(\sum_{j=1}^{3}g^{ij,n}\frac{\partial q^n}{\partial
y_j}\right)=-\di (M^n \P v^n+N^n
\P v^n), \, \,& & \text{in}\,\, \Omega\\
&\sum_{i,j=1}^3 g^{ij,n}\frac{\partial q^n }{\partial y_j}n_i
+\frac{1}{m}\left(\int_{\partial\Omega} q^n \vec{n}\dif\sigma\right)
\cdot
\vec{n}+\left[\bar{J}^{-1}\int_{\partial\Omega}y\times q^n \vec{n}\dif\sigma\right]\times y\cdot \vec{n}\\
&=-(M^n \P v^n+N^n \P v^n)\cdot
\vec{n}- (R^n\times L^n)\cdot \vec{n} +
\left[\bar{J}^{-1}(\bar{J}R^n\times R^n)\right]\times
y\cdot \vec{n}, \,\, && \text{on}\,\,\partial\Omega.
\end{aligned}\right.\label{5.2}
\end{equation}

Now define an operator $A^n(t,v)$ as in \cite{RR},
\begin{equation}
\begin{aligned}
&A^n(t,v)=\Vec{1}_{\Omega}\left(\frac{\p Y^{n-1} }{\p t}+\P
v^{n-1}\right)\cdot \nabla v
 -\mathcal{Q}\left[\Vec{1}_{\Omega}\left(\frac{\p Y^{n-1} }{\p t}+\P v^{n-1}\right)
 \cdot \nabla \P v\right]\\
& + \P\left[\vec{1}_{\Omega}
\left(\sum_{j,k=1}^3\left\{\Gamma_{j,k}^{\cdot ,n-1}\frac{\p
Y_k^{n-1}}{\p t} +\frac{\p Y^{n-1}}{\p x_k}\frac{\p^2 X_k^{n-1}}{\p
t\p y_j}
\right\}(\P v)_j^{n-1}+\sum_{j,k=1}\Gamma_{j,k}^{\cdot ,n-1}
(\P v)_j^{n-1} (\P v)_k^{n-1} \right)\right]\\
&\P \left[\vec{1}_{\Omega}\left(\sum_{j=1}^3 g^{\cdot
j,n-1}\frac{\p q^{n-1}}{\p y_j}\right)\right]
+\P \left[\vec{1}_{\mathcal{O}}\left(R^{n-1}\times
L^{n-1}-\bar{J}^{-1}(\bar{J} R^{n-1}
\times R^{n-1})\times y\right)\right]\\
& + \P \left[\vec{1}_{\mathcal{O}}\left(-\frac{1}{m}
\int_{\partial\Omega} q^{n-1}\vec{n}\dif\sigma-\left(\bar{J}^{-1}
\int_{\partial\Omega}y\times q^{n-1} \vec{n}\dif\sigma\right)\times
y\right)\right],
\end{aligned}\label{5.3}
\end{equation}
where the operator $\mathcal{Q}= I-\P$.

Let
$$v_0(y) = \left\{\begin{array}{l}
u_0(y), \ \ \ \ \ y\in \Omega,\\
l_0 + \omega_0\times y,\ \ \ y\in \mathcal{O}.\end{array}\right.$$
Consider the following Cauchy problem,
\begin{equation}
\left\{\begin{aligned}
v^n_t+A^n(t,v^n)=0,\\
v^n(0)=v_0 ,
\end{aligned}\right.\label{5.4}
\end{equation}
 where $v_0\in H_s\cap\tilde{X}_*.$

In particular, let $v^0(y,t) = v_0(y)$. We shall prove that for each $n\in \mathbb{N}$, there exists a solution
$v^n\in C_w(0,T_n; H_s)$ with some uniform lifespan $T_n$.

For simplicity,  denote
\begin{alignat*}{12}
(F^{n-1}v)_i&=\left(\frac{\p Y^{n-1} }{\p t}+ \P
 v^{n-1}\right)\cdot \nabla v_i, \,\, (G^{n-1}\cdot \nabla q^{n-1})_i
 =\sum_{j=1}^3 g^{ij,n-1}\frac{\p q^{n-1}}{\p y_j},\\
(E^{n-1})_i&=\vec{1}_{\Omega}\left[\left(\sum_{j,k=1}^3
\left\{\Gamma_{j,k}^{i,n-1}\frac{\p Y_k^{n-1}}{\p t}+\frac{\p
Y_i^{n-1}}{\p x_k}
\frac{\p^2 X_k^{n-1}}{\p t\p y_j}\right\}(\P v^{n-1})_j\right.\right.\\
&\left.\left.+\sum_{j,k=1}\Gamma_{j,k}^{i,n-1} (\P v^{n-1})_j (\P v^{n-1})_k\right)\right],\\
(K^{n-1})_i&=\left[\vec{1}_{\mathcal{O}}\left(\frac{1}{m}R^{n-1}\times L^{n-1}-\bar{J}^{-1}(\bar{J}R^{n-1}
\times R^{n-1})\times y\right)\right]_i\\
&-\left\{\vec{1}_{\mathcal{O}}\left[\frac{1}{m}\int_{\partial\Omega}
q^{n-1}\vec{n}
\dif\sigma-\left((\bar{J}^{-1}\int_{\partial\Omega}y\times q^{n-1}
\vec{n}\dif\sigma)\times y\right)\right]\right\}_i,\\
\end{alignat*}  and denote $M_0 = 2\|v_0\|_{H_{s}}$. Suppose that there exists some $T_0>0$ such
that for all $k<n$,
$$\|v^{k}\|_{L^{\infty}(0,T_0;H_s)}\leq M_0.$$

For the estimate of $( v, A^n(t,v))_{H_s}$,
\begin{equation*}
\begin{aligned}
|( v,A^n(t,v))_{H_s} |&\le
|(\vec{1}_{\Omega}F^{n-1}v,v)_{H_s}|+|(\mathcal{Q}(\vec{1}_{\Omega}F^{n-1}\P v),v)_{H_s}|
+|(\P E^{n-1}v,v)_{H_s}|\\
&+|(\P(\vec{1}_{\Omega}G^{n-1}\nabla q^{n-1}),v)_H|+|(\P K^{n-1},v)_{H_s}|\\
&:= J_1+J_2+J_3+J_4+J_5
\end{aligned}
\end{equation*}

Then we estimate them term by term. Starting from the easiest one,
\begin{equation}\begin{aligned}J_5 & \le\|\P K^{n-1}\|_{H_s}\|v\|_{H_s}\\
& \leq C\| K^{n-1}\|_{H_s}\|v\|_{H_s}  \\  & \le
C\left(\left|\int_{\partial \Omega} q^{n-1}
\vec{n}\dif \sigma\right|+\left|\int_{\partial \Omega}y\times q^{n-1}\vec{n} \dif \sigma\right|
+|L^{n-1}||R^{n-1}|+|R^{n-1}|^2\right)\|v\|_{H_s}\\
& \leq C\left(|L^{n-1}| \cdot |R^{n-1}|+ |R^{n-1}|^2 + \|\nabla q^{n-1}\|_{L^2(\Omega)}\right)\cdot \|v\|_{H_s}\\
& \leq C(T_0, M_0)\|v\|_{H_s}.
\end{aligned}\label{5.5}\end{equation}

By Lemma 3.3 and Lemma 3.4,
\begin{equation}\label{5.6}
\begin{aligned}J_4 & \le \|\P(\vec{1}_{\Omega}(G^{n-1}\nabla
q^{n-1}))\|_{H_s}\cdot \|v\|_{H_s} \\ & \leq C \|G^{n-1}\nabla
q^{n-1}\|_{H^s(\Omega)}\cdot \|v\|_{H_s} \\ &\le C(T_0,M_0)\|v\|_{H_s}.
\end{aligned}
\end{equation}

$J_3$ is also easy to estimate since there is no derivative of $v$
or $v^{n-1}$, \begin{equation}\label{5.7} J_3\le
C(T_0,M_0)\|v\|_{H_s}.\end{equation}

Now the most difficult terms $J_1$ and $J_2$ are left, since there
is derivative of $v$ or $\P v$.
\begin{equation}
\begin{aligned}
J_1= \left|\sum_{|\alpha|\le s} \sum_{\alpha_1 \leq \alpha}
\sum_{i=1}^3 \int_{\Omega}\p^{\alpha_1}\left(\frac{\p Y^{n-1} }{\p
t}+\P v^{n-1}\right)\cdot \nabla\p^{\alpha - \alpha_1}
v_i\p^{\alpha}v_i dy\right|.
\end{aligned}\label{5.8}
\end{equation}
When $\alpha_1=(0, 0, 0)$, since $$\di\left(\frac{\p Y^{n-1} }{\p
t}+\P v^{n-1}\right)=0 \ \ \mbox{in}\ \Omega,\ \ \ \mbox{and}\ \  \ \left(\frac{\p Y^{n-1}
}{\p t}+\P v^{n-1}\right)\cdot \vec{n}=0\ \ \mbox{on}\ \p\Omega,$$
then
$$\int_{\Omega}\left(\frac{\p Y^{n-1} }{\p t}+\P
 v^{n-1}\right)\cdot \nabla\p^{\alpha} v_i\p^{\alpha}v_i dy =0.$$

Therefore,  we assume that $|\alpha_1|\geq 1$. Let $\alpha_2 = \alpha - \alpha_1$.
\begin{equation*}
\begin{aligned}
&\left\|\p^{\alpha_1}\left(\frac{\p Y^{n-1} }{\p t}+ \P
 v^{n-1}\right)\cdot \nabla\p^{\alpha_2} v_i\right\|_{L^2(\Omega)}\\
&\le \left\|\p^{\alpha_1}\frac{\p Y^{n-1} }{\p t}\cdot
\nabla\p^{\alpha_2} v_i\right\|_{L^2(\Omega)}
+ \left\|\p^{\alpha_1}\P v^{n-1}\cdot \nabla\p^{\alpha_2} v_i\right\|_{L^2(\Omega)}\\
&\le
C(T_0,M_0)\|v\|_{H_s}+\|\P v^{n-1}\|_{L^{\infty}(\Omega)}\|\nabla
v\|_{H^{s-1}(\Omega)}
+\|\nabla v\|_{L^{\infty}(\Omega)}\|\P v^{n-1}\|_{H^{s}(\Omega)}\\
&\le C(T_0,M_0)\|v\|_{H_s}.
\end{aligned}
\end{equation*}

Hence, \begin{equation}\label{5.9} J_1\le
C(T_0,M_0)\|v\|_{H_s}^2.\end{equation}

For the term $J_2$,
\begin{alignat*}{11}
&J_2\le\|\mathcal{Q}[\vec{1}_{\Omega} F^{n-1}
\P v]\|_{H^s(\Omega)}
\|v\|_{H^s(\Omega)}+C\|F^{n-1} \P v\|_{L^2(\Omega)}\|v\|_{L^2(\mathcal{O})}.
\end{alignat*}
Herein,
\begin{alignat*}{11} \|F^{n-1} \P v\|_{L^2(\Omega)}
\|v\|_{L^2(\mathcal{O})} & \le\left\|\frac{\p Y^{n-1}}{\p
t}+\P v^{n-1}
\right\|_{L^\infty(\Omega)}\|\P v\|_{H_{s}}\|v\|_{L^2(\mathcal{O})}\\
&\le C(T_0,M_0)\|v\|_{H_s}^2.
\end{alignat*}

To estimate $\|\mathcal{Q}[\vec{1}_{\Omega} F^{n-1}
\P v]\|_{H^s(\Omega)}$, consider the following system,
\begin{equation*}
\left\{\begin{aligned} \Delta\phi=\di\left(\frac{\p Y^{n-1}}{\p t}+
\P v^{n-1}\right)\cdot\nabla \P v,
 \,\,\text{in}\,\, \Omega ,\\\
\frac{\p \phi}{\p n}=\left(\frac{\p Y^{n-1}}{\p t} +\P
v^{n-1}\right)\cdot\nabla \P v\cdot \vec{n}.\,\, \text{on}\,\,
\p\Omega.
\end{aligned}\right.
\end{equation*}
In fact, $\mathcal{Q}[\vec{1}_\Omega F^{n-1}\P v]= \nabla \phi.$

Note that in the domain $\Omega$,
$$\di \left[\left(\frac{\partial Y^{n-1}}{\partial t} + \P v^{n-1}\right)\cdot \nabla \P v \right]
= \sum_{i, j=1}^3 \partial_j \left(\frac{\partial Y^{n-1}}{\partial t}+ \P v\right)_i \partial_i (\P v)_j,$$
and on the boundary $\partial \Omega$,
$$\begin{array}{ll}\left(\frac{\partial Y^{n-1}}{\partial t}+ \P v^{n-1}\right)\cdot \nabla \P v\cdot \vec{n}
= & \left(\frac{\p Y^{n-1}}{\partial t }+ \P v^{n-1}\right)\cdot \nabla \left(
\P v - l_{\P v} - \omega_{\P v}\times y\right)\cdot \vec{n}\\
& + \left(\frac{\partial Y^{n-1}}{\partial t} + \P v^{n-1}\right)\cdot \nabla
\left(l_{\P v} + \omega_{\P v}\times y\right)\cdot \vec{n}.
\end{array}$$
Hence, as estimating $\nabla q$ in section 4, one can get that
$$\|\nabla \phi\|_{H^s(\Omega)}\le C(T_0,M_0)\|v\|_{H_s},$$
consequently,
$$J_2 \leq C(T_0, M_0)\|v\|_{H_s}^2.$$

Therefore, we have
\begin{equation}\label{5-add}
|(v,A^n(t,v))_{H_s}|\le C(T_0,M_0) (1+\|v\|_{H_s}^2).
\end{equation}



Now, fix some big $T_0$. Let's consider the corresponding ordinary differential equation,
$$\gamma^\prime(t) = C(T_0, M_0) (1 + \gamma(t)),\ \ \ \gamma(0) = \|v_0\|_{H_s}^2,$$
with $C(T_0, M_0)$ the same constant as in (\ref{5-add}). Assume that $T\leq T_0$ is a time such that for every
$t\in [0, T]$,
$$\gamma(t) \leq 4 \|v_0\|_{H_s}^2 = (M_0)^2.$$
Then by the Kato-Lai theory, the solution $v^n$ to (\ref{5.4}) can be derived at least on $[0, T]$, i.e., $v^n\in C_w([0, T]; H_s)\cap C_w^1([0, T]; \tilde{X})$,
and
$$\|v^n(t)\|_{H_s}^2 \leq \gamma(t)\leq (M_0)^2, \ \ \ \ t\in [0, T].$$

For $n=1$, we choose $v^0(y,t) = v_0(y)$. Following the above process, one can construct a solution
$v^1$ to the system (\ref{5.4}). By iterating the same steps, a sequence of approximate solutions $\{v^n\}$
can be constructed.

 \section{The convergence of approximate solutions }

 In this section, we show that $\{v^n\}$ converges to a solution of the system (\ref{2.16})-(\ref{2.22}).

According to the estimates in section 5,
\begin{equation}\label{6.1}\|v^n\|_{L^{\infty}(0,T; H_s)}\le M_0 ,
\end{equation}
\begin{equation}\label{6.2}
\|\partial_t v^n\|_{L^{\infty}(0,T; \tilde{X})}\le M_1,\end{equation}
\begin{equation}\label{6.3} \|\nabla q^n \|_{L^{\infty}(0,T; H^{s-1}(\Omega))}\leq
M_2.\end{equation}

Since $\P$ is a bounded operator on $H_s$, $\tilde{X}$, and
it commutes with $\partial_t$, then

\begin{equation}\label{6.4}
\|\P v^n\|_{L^\infty(0, T; H_s)}\leq M_3,
\end{equation}
\begin{equation}\label{6.5}
\|\partial_t \P v^n\|_{L^\infty(0, T; \tilde{X})}\leq M_4.
\end{equation}

Hence from the \cite{S}, there exists some function $v\in C_w([0,T]; H_s)$ such that,
\begin{equation}
v^n \rightarrow v\,\, \mbox{in}\,\  C_w(0,T; H_s),\label{6.6}
\end{equation}
\begin{equation}\label{6.7}
\P v^n \rightarrow \P v\ \  \mbox{in}\ \ C_w(0, T;
H_s).
\end{equation}

By the Aubin-Lions lemma, for every $r_0$ large enough,
\begin{equation}\label{6.8}
v^n \to v \ \ \mbox{in}\ \ C([0, T]; H^{s-1}(\Omega_{r_0})\cap
L^2(B_{r_0})),
\end{equation}
\begin{equation}\label{6.9}
\P v^n \to \mathbb{P}v \ \ \mbox{in}\ \ C([0, T];
H^{s-1}(\Omega_{r_0})\cap L^2(B_{r_0})).
\end{equation}

Moreover, (\ref{6.9}) implies that
\begin{equation}\label{6.10}
L^n(t) \to L(t)=l_{\P v(t)} \ \ \mbox{in}\ \ C[0, T],
\end{equation}
\begin{equation}\label{6.11}
R^n(t) \to R(t)= \omega_{\P v(t)}\ \ \mbox{in}\ \ C[0, T].
\end{equation}

While (\ref{6.3}), (\ref{6.8}) and (\ref{6.9}) tell that there exists some function $q$ such
that
\begin{equation}\label{6.13}
\nabla q^n \rightarrow \nabla q \ \ \mbox{in}\ \  C_w([0, T];
H^{s-1}(\Omega)),
\end{equation}
\begin{equation}\label{6.14}
\int_{\partial \Omega}q^{n} \vec{n} d\sigma \to \int_{\partial
\Omega}q\vec{n} d\sigma\ \ \mbox{in}\ \ C[0, T],
\end{equation}
\begin{equation}\label{6.15}
\int_{\partial\Omega}y\times q^n \vec{n}d\sigma \to
\int_{\partial \Omega}y\times q\vec{n} d\sigma\ \ \mbox{in}\ \
C[0, T].\end{equation} In fact, $q$ is a solution to the system
(\ref{4.1}). It can be seen by taking the limit of (\ref{5.2}).

From all the convergence results (\ref{6.1})-(\ref{6.15}), it follows that
\begin{equation}\label{6.16}
\left\{\begin{aligned}
&v_t+A(t,v)=0,\\
&v(0)=v_0,
\end{aligned}\right.
\end{equation}
where
\begin{equation*}
\begin{aligned}
&A(t,v)=\Vec{1}_{\Omega}\left(\frac{\p Y }{\p t}+ \P
v\right)\cdot \nabla v -\mathcal{Q}\left[\Vec{1}_{\Omega}
\left(\frac{\p Y }{\p t}+\P  v\right)\cdot \nabla \P v\right]\\
& + \P\left[\vec{1}_{\Omega} \left(\sum_{j,k=1}
\left\{\Gamma_{j,k}^{\cdot}\frac{\p Y_k}{\p t}+\frac{\p Y}{\p x_k}
\frac{\p^2 X_k}{\p t\p y_j}\right\}(\P v)_j+
\sum_{j,k=1}\Gamma_{j,k}^{\cdot} (\P v)_j (\P v)_k  \right)\right]\\
& + \P \left[\vec{1}_{\Omega}\left(\sum_{j=1}^3 g^{\cdot j}
\frac{\p q}{\p y_j}\right)\right]+
\P\left[\vec{1}_{\mathcal{O}}\left(R\times L+
\bar{J}^{-1}(\bar{J} R\times R)\times y\right)\right]\\
&
-\P\left[\vec{1}_{\mathcal{O}}\left(\frac{1}{m}\int_{\partial\Omega}
q\vec{n}\dif\sigma-\left(\bar{J}^{-1}\int_{\partial\Omega}y\times
q \vec{n}\dif\sigma\right)\times y\right)\right],
\end{aligned}
\end{equation*}

Next, we shall prove that $v$ is a  solution of the systems
(\ref{2.16})-(\ref{2.22}). The proof starts with the observation
that $v(t)=\P v(t)$, for all $t\in [0,T]$. In fact,
applying $\mathcal{Q}$ to each term in (\ref{6.16}) and taking the
inner product with $\mathcal{Q}v(t)$ in $\tilde{X}$ yields
\begin{equation}
\frac{\dif}{\dif
t}\frac{1}{2}\|\mathcal{Q}v(t)\|_{\tilde{X}}^2+(\mathcal{Q}v(t),
\mathcal{Q} A(t,v))_{\tilde{X}}=0. \label{6.17}\end{equation}

Note that $\di \left(\frac{\p Y }{\p t}+\P v\right)=0$ in $\Omega$
and $\left(\frac{\p Y }{\p t}+\P v\right)\cdot \vec{n}=0$ on
$\p\Omega$, then
\begin{equation}\label{6.18}
\begin{aligned}
(\mathcal{Q}v(t), \mathcal{Q} A(t,v))_{\tilde{X}}&=\left(\mathcal{Q}v(t), \vec{1}_{\Omega}\left(\frac{\p Y }{\p t}+\P v\right)\cdot \nabla \mathcal{Q}v(t) \right)_{\tilde{X}}\\
&=\int_{\Omega}\mathcal{Q}v\cdot\left(\left(\frac{\p Y }{\p t}+\P v\right)\cdot \nabla \mathcal{Q}v(t)\right) \dif y\\
&=0
\end{aligned}
\end{equation}

Since $v_0= \P v_0$, it tells that $\mathcal {Q} v_0=0.$
Hence, for every $t\in [0, T]$, $\mathcal{Q}v(t)=0$. Therefore,
(\ref{6.16}) can be written as
\begin{equation}\label{6.19}
\begin{aligned}
&\frac{\p v}{\p t}+\P[\vec{1}_\Omega(Mv+Nv+G\cdot\nabla q)]+\\
&\P\left[\vec{1}_{\mathcal{O}}\left(\frac{1}{m}R\times L-\bar{J}^{-1}(\bar{J}R\times R)\times y\right)\right.\\
& -\left.\vec{1}_{\mathcal{O}}\left(\frac{1}{m}\int_{\partial\Omega}
q\vec{n}\dif\sigma-(\bar{J}^{-1}\int_{\partial\Omega}y\times q
\vec{n}\dif\sigma)\times y\right)\right]=0
\end{aligned}
\end{equation}

Taking the inner product in $\tilde{X}$ with a test function
$\phi\in\tilde{X}_*$, one has
\begin{equation}
\begin{aligned}
&\int_{\Omega} (v'+Mv+Nv+G\cdot\nabla q)\cdot \phi\dif y \\
& + m L'\cdot l_{\phi} -\int_{\partial\Omega} q\vec{n}\dif\sigma\cdot
l_{\phi}+m (R \times L)
\cdot l_{\phi}\\
&+ \bar{J}R'\cdot\omega_\phi-\bar{J}\left[\bar{J}^{-1}(\bar{J}R\times R)\right]
\cdot\omega_\phi-\bar{J}\left[\bar{J}^{-1}\int_{\partial\Omega}y\times
q \vec{n}\dif\sigma
\right]\cdot\omega_\phi\\
&=0.
\end{aligned}\label{6.20}
\end{equation}

For every function $\phi\in C_0^\infty(\mathbb{R}^3)$, with
$supp(\phi) \subseteq \Omega$, and ${\rm div}~\phi =0$ in
$\mathbb{R}^3$, (\ref{6.20}) yields
$$\int_{\Omega} (v'+Mv+Nv+G\cdot\nabla q)\cdot \phi \dif y=0,$$

After the theory of Helmholtz-Weyl decomposition, there exists a function
$p$ such that $\nabla p\in L^{\infty}(0,T; H^{s-1}(\Omega))$ and
\begin{equation} v'+Mv+Nv+G\cdot\nabla q+\nabla p=0\,\,\ \text{in}\,\,\
\Omega\times[0,T].\label{6.21}\end{equation}

From the identification of $q$ and (\ref{6.21}), one knows that for
every $t\in [0, T]$,
\begin{equation*} \left\{\begin{aligned}
&\Delta p=0, \,\, &\text{in}\,\,\Omega,\\
&\frac{\p p}{\p \vec{n}}=0,\,\, &\text{on}\,\,\p\Omega.
\end{aligned}\right.
\end{equation*}

The above system has only constant solutions, thus
\begin{equation}\label{6.22}
v'+Mv+Nv+G\cdot\nabla q=0\ \ \ \text{in}\ \ \Omega\times[0,T].
\end{equation}

Now taking some test function $\phi(x)\in \tilde{X}$ such that
$\phi(y)=l_{\phi}$ in $\mathcal{O}$, then
\begin{equation*}
m L'\cdot l_{\phi} - \left(\int_{\partial\Omega}
q\vec{n}\dif\sigma\right) \cdot l_{\phi}+ (m R\times L) \cdot
l_{\phi}=0.
\end{equation*} Since $l_\phi$ is arbitrary, then
\begin{equation}\label{6.23} mL^\prime = \int_{\partial \Omega}
q\vec{n} \dif\sigma - R\times L.\end{equation}

Similarly, taking some test function $\phi(y)\in \tilde{X}$ such
that $\phi(y) = \omega_{\phi}\times y$ in $\mathcal{O}$, then
\begin{equation*}
\bar{J}R'\cdot\omega_\phi-\bar{J}(\bar{J}^{-1}(\bar{J}R\times R))
\cdot\omega_\phi-\bar{J}\left(\bar{J}^{-1}\int_{\partial\Omega}y\times q
\vec{n}\dif\sigma\right)\cdot\omega_\phi=0.
\end{equation*}

Thus
\begin{equation}\label{6.24}\bar{J}R'=\int_{\partial\Omega}y\times q
\vec{n}\dif\sigma+\bar{J}R\times R.\end{equation}

Therefore, $(v, q, L(t), R(t))$ is a solution to the system (\ref{2.16})-(\ref{2.22}).

\section{Uniqueness and continuity with respect to time}
In this section, we will prove that the solution of
the system (\ref{2.16})-(\ref{2.22}) is unique and then get the
continuity in $H_s$ with respect to time.

Assume that there exist two solutions $v^1,v^2\in C_{\rm w}([0,T];
H_s)\cap C^1_{\rm w}([0,T];\tilde{X})$ to the system
(\ref{2.16})-(\ref{2.22}), then
\begin{equation}\label{7.1}
v^1_t+M^1v^1+N^1v^1+G^1\nabla q^1=0,\ \ \text{in}\ \
\Omega\times[0,T],
\end{equation}
\begin{equation}\label{7.2}
v^2_t+M^2v^2+N^2v^2+G^2\nabla q^2=0,\ \ \text{in}\ \
\Omega\times[0,T].\end{equation}

Let $H^1 = (G^1)^{-1}$ and $H^2 =(G^2)^{-1}.$
Multiplying (\ref{7.1}) and (\ref{7.2}) by $H^1$ and $H^2$
respectively, and denote $K=\max\{\|v^1\|_{L^\infty(0,T;H_s)},\|v^2\|_{L^\infty(0,T;H_s)}\}$.

Subtracting the two equations and taking inner product
in $L^2(\Omega)$ with $v^1-v^2$, then one gets
\begin{equation*}
\begin{aligned}
&0=(H^1v^1_t-H^2v^2_t,v^1-v^2)_{L^2(\Omega)}+(\nabla q^1-\nabla q^2,v^1-v^2)_{L^2(\Omega)}\\
&\ \ \ +(H^1(M^1v^1+N^1v^1)-H^2(M^2v^2+N^2v^2),v^1-v^2)_{L^2(\Omega)}\\
&\ \ :=I_1+I_2+I_3,
\end{aligned}
\end{equation*}

Denote $l_{v^1}, \omega_{v^1}, l_{v^2}, \omega_{v^2}$ by $L^1,
R^1, L^2, R^2$ respectively.
 \begin{equation}\label{7.3}\begin{aligned}
 I_2 & =\displaystyle\int_{\Omega}\nabla(q^1-q^2)\cdot (v^1-v^2)\dif y
 \\
  & =\displaystyle\int_{\p\Omega}(q^1-q^2)(v^1-v^2)\cdot \vec{n}\dif\sigma
  \\
 &=\displaystyle\int_{\p\Omega}(q^1-q^2)(L^1-L^2)\cdot \vec{n} \dif \sigma
 +\int_{\p\Omega}(q^1-q^2)(R^1-R^2)\times y\cdot \vec{n}\dif\sigma
 \\
  &=\displaystyle m(L^1-L^2)'\cdot (L^1-L^2)+ m(R^1 \times L^1- R^2
  \times L^2)\cdot (L^1 - L^2) \\
   &\ \ + \displaystyle \bar{J}(R^1  - R^2)^\prime \cdot (R^1 - R^2)
   -(\bar{J}R^1\times R^1-\bar{J}R^2\times  R^2)\cdot(R^1- R^2)\\
 & = \displaystyle \frac{1}{2}m\frac{\dif}{\dif t}|L^1-L^2|^2+
 \frac{1}{2}\frac{\dif}{\dif
 t}\left[(\bar{J}(R^1- R^2))\cdot(R^1- R^2)\right]\\
 &\ \ +\displaystyle m(R^1\times L^1- R^2\times L^2)\cdot(L^1- L^2)-
 (\bar{J} R^1\times R^1-\bar{J}R^2\times R^2)\cdot(R^1- R^2).\end{aligned}
 \end{equation}

The term $I_1$ can be estimated as follows,
 \begin{equation}\label{7.4}\begin{aligned}
I_1 & =(H^1v^1_t-H^2v^2_t,v^1-v^2)_{L^2(\Omega)}\\
&=\left(H^1(v^1-v^2)_t,v^1-v^2\right)_{L^2(\Omega)}+\left((H^1-H^2)v^2_t,v^1-v^2\right)_{L^2(\Omega)}\\
&:=I_{11}+I_{12}. \end{aligned}\end{equation}

From the definition of $G$,
\begin{equation}\label{7.5}\begin{aligned}
I_{11} & =\left(H^1(v^1-v^2)_t,\  v^1-v^2\right)_{L^2(\Omega)}\\
&=\left(J^T_{X^1}J_{X^1}(v^1-v^2)_t,\ v^1-v^2\right)_{L^2(\Omega)}\\
&=\left(J_{X^1}(v^1-v^2)_t,\ J_{X^1}(v^1-v^2)\right)_{L^2(\Omega)}\\
&=\displaystyle \frac12\frac{\dif}{\dif t}\left(J_{X^1}(v^1-v^2),\
J_{X^1}(v^1-v^2)\right)_{L^2(\Omega)} \\[3mm] & \ \ \ - \displaystyle \left(\frac{\p
J_{X^1}}{\p t}(v^1 - v^2),\
J_{X^1}(v^1-v^2)\right)_{L^2(\Omega)}.\end{aligned}
\end{equation}
Therefore,
\begin{equation}\label{7.6} \displaystyle
I_{11} \ge \frac12 \frac{\dif}{\dif
t}\left\|J_{X^1}(v^1-v^2)\right\|_{L^2(\Omega)}^2
 -C(T, K)\sup_{s\in [0,t]}\|v^1(s)-v^2(s)\|_{L^2(\Omega)}^2.
 \end{equation}

On the other hand, by Lemma 3.5,
 \begin{equation}\label{7.7}\begin{aligned}
|I_{12}| & \le C\|g^1_{ij}-g^2_{ij}\|_{L^\infty(\Omega)}
\|v^1-v^2\|_{L^2(\Omega)}\|v^2_t\|_{L^2(\Omega)}\\
&\leq \displaystyle C(T, K)\left(\sup_{s\in
[0,t]}(|l^1(s)-l^2(s)|+|\omega^1(s)-\omega^2(s)|)\right) \sup_{s\in
[0,t]}\|v^1(s)-v^2(s)\|_{L^2(\Omega)}\\
&\leq\displaystyle C\sup_{s\in
[0,t]}\|v^1(s)-v^2(s)\|^2_{L^2(\mathbb{R}^3)}.
\end{aligned}\end{equation}

$I_3$ can be estimated similarly,
$$|I_3|\leq C\sup_{[0,t]}\|v^1(s)-v^2(s)\|^2_{L^2(\mathbb{R}^3)}.$$

Therefore,
\begin{equation*}
\begin{aligned}
&\frac{\dif }{\dif t}\|J_{X^1}(v^1-v^2)\|^2_{L^2(\Omega)}+
m\frac{\dif }{\dif t}|L^1-L^2|^2+
\frac{\dif }{\dif t}[(\bar{J}(R^1- R^2))\cdot(R^1- R^2)]\\
&\le C\sup_{[0,t]}\|v^1(s)-v^2(s)\|^2_{L^2(\mathbb{R}^3)}.
\end{aligned}
\end{equation*}
It follows that
\begin{equation*}
\begin{aligned}
&\|J_{X^1}(v^1-v^2)(t)\|^2_{L^2(\Omega)}+m|L_1-L_2|^2+(\bar{J}(R_1- R_2))\cdot(R_1-  R_2)\\
&\le C \int_0^t\sup_{s\in
[0,\tau]}\|v^1(s)-v^2(s)\|^2_{L^2(\mathbb{R}^3)}\dif\tau
\end{aligned}
\end{equation*}
Since $X^1$ is always a diffeomorphism and $\bar{J}$ is positive definite, thus
\begin{equation}\label{7.15}
\|v^1-v^2\|^2_{L^2(\mathbb{R}^3)}(t)\le C \int_0^t\sup_{\tau\in [0,s]}\|v^1-v^2\| ^2_{L^2(\mathbb{R}^3)}(\tau)\dif s
\end{equation}

 By the Gronwall's inequality, $v^1=v^2\,\, a.e.\,\,$ in $[0,T]\times\mathbb{R}^3.$
 Uniqueness, as in \cite{RR} or earlier publication \cite{Temam} for the Euler equations, combining the fact that
the system is reversible, implies
 $$v\in C([0,T];H_s)\cap C^1([0,T];H_{s-1}).$$

In fact, from the preceding estimates, one can easily get that
 \begin{equation}\label{7.8}\frac{d}{dt}\|v\|_{H^s(\Omega)}^2 + m\frac{d}{dt}|L|^2 + \frac{d}{dt}\left[\bar{J}R \cdot R\right]
 \leq C \|v\|_{H_s}^2 \left(\|\nabla v\|_{L^\infty(\Omega)} + \|v\|_{L^2(\R^3)} + 1\right).
\end{equation}

It implies that once $\|\nabla v\|_{L^\infty(\Omega)}$ does not blow up,  $\|v\|_{H_s}$ will not blow up. Using the argument as in the paper \cite{RR}, one can get that the lifespan of the solution does not depend on $s$.

\textbf{Acknowledgements}

Gratitude is expressed specially to supervisor Professor Zhouping Xin and the referees for their careful reading of the manuscript and their fruitful suggestions.

\end{document}